\documentclass[12pt,twoside,a4paper]{amsart}
\usepackage[latin1]{inputenc}
\usepackage{hyperref}
\usepackage[all]{xy}
\usepackage{amssymb}

\def\ol#1{{\overline{#1}}}
\def\wt#1{{\widetilde{#1}}}
\title [Weil-Petersson geometry ]{Weil-Petersson geometry for
families of hyperbolic conical  Riemann Surfaces}
\author{Georg Schumacher\and Stefano Trapani}
\date  {} 

\newtheorem{theorem}             {Theorem}    [section]
\newtheorem{corollary}  [theorem]{Corollary}
\newtheorem{lemma}      [theorem]{Lemma}

\newtheorem{proposition}[theorem]{Proposition}
\newtheorem{definition}    [theorem]{Definition}

\newtheorem{remark}    [theorem]{Remark}      %

\newenvironment{theorem*}[1]{{\bf Theorem.(#1)} \begin{it}}
{\end{it}}

\newcounter    {assertcount}

\newcommand{\cT}{\mathcal{T}}
\newcommand{\cX}{\mathcal{X}}
\newcommand{\cK}{\mathcal{K}}

\newcommand{\pt}{\partial}
\newcommand{\ii}{\sqrt{-1}}

\newcommand{\C} {{\mathbb  C}}
\newcommand{\N} {{\mathbb  N}}
\newcommand{\R} {{\mathbb  R}}
\newcommand{\Z} {{\mathbb  Z}}
\newcommand{\Q} {{\mathbb  Q}}
\newcommand{\ga}{g_{\bf a}}
\newcommand{ \A }{{\bf a}}

\renewcommand{\epsilon}{\varepsilon}

\def\wp{Weil-Pe\-ters\-son }
\def\tei{Teich\-m\"{u}l\-ler }
\def\cT{\mathcal{T}}

\def\mapa#1{{\vrule height 10pt width 2pt depth 0pt}\marginpar{\begin{sloppy}
\footnotesize #1 \end{sloppy}}}

\begin{document}

\begin{abstract}
We study the \wp geometry for holomorphic  families of  Riemann Surfaces
equipped with  the unique conical metric of constant curvature $-1$.
\end{abstract}

\maketitle

\section{Introduction} \label{Int}

Recently hyperbolic structures on weighted punctured Riemann surfaces
gained major attention. Hyperbolic metrics on weighted punctured Riemann
surfaces by definition have conical singularities at the punctures, where
the cone angles are between $0$ and $2 \pi$, corresponding to weights
between one and zero. Conical metrics of constant negative curvature (with
fixed weights) induce new structures on the \tei spaces of punctured
Riemann surfaces. Tan, Wong and Zhang \cite{ T-W-Z} showed the existence
of corresponding Fenchel-Nielsen coordinates, proved a McShane identity
for this case and investigated the induced symplectic structure. In this
way they generalize results of Mirzakani \cite{ Mz} to this situation
(cf.\ \cite{D-N}). Conical metrics on punctured spheres were studied by
Zograf and Takhtajan in \cite{TaZo}, who introduce Kähler structures on
the moduli spaces depending on cone angles in the context of Liouville
actions. From the algebraic geometry point of view, Hassett \cite{Ha}
introduced a hierarchy of compactifications of the moduli space of
punctured Riemann surfaces according to the assigned weights of the
punctures. These spaces interpolate between the classical Deligne-Mumford
compactifications of the moduli spaces of Riemann surfaces with and
without punctures. Conical hyperbolic metrics had been studied  by Heins
\cite{h}, and constructed by McOwen \cite{Mc-O} and Troyanov \cite{Tb}
using the method of Kazdhan and Warner \cite{ K-Wc}.

By definition, a weighted punctured Riemann surface $(X, {\bf a})$ is a
compact Riemann surface $X$ together with an an $\R$-divisor ${\bf a} =
\sum_{j=1}^{n} a_j p_j$ with weights $0<a_j\leq 1$ at the punctures $p_j$.
The necessary and sufficient condition for the existence of a hyperbolic
conical metric according to \cite{Mc-O,Tb}  is that the statement of the
Gauss-Bonnet theorem holds, i.e.\ the degree of $K_X + { \bf a}$ is
positive, where $K_X$ denotes the canonical divisor of $X$. In this case
the cone angles are $2 \pi (1- a_j)$.

Our aim is to study the \wp geometry in the conical case, and develop a
theory parallel to the classical one. We show the existence of a \wp
K\"{a}h\-ler form of class $C^{\infty}$,  which descends to to the moduli
space. Let $\mathcal{X} \rightarrow S$ be the universal family, or any
other holomorphic family of weighted punctured Riemann surfaces. It turns
out that the  classical Wolpert's formula, \cite[Corollary 5.11]{Wo},
holds in our case as well, i.e.\ the \wp form is the push forward of the
form $ 2 \pi^2 c_1({K_\mathcal{{X}/S}}^{-1},g_{\bf{a}})$, where
$({K_\mathcal{{X}/S}}^{-1},g_{\bf{a}})$ is the relative anti-canonical
line bundle, equipped with the family of hyperbolic conical metrics on the
fibers. From this we derive the K\"{a}h\-ler property of the \wp metric.

For rational weights the bundle $K_{\cX/S}+ {\bf a}$ defines a determinant
line bundle on the base space $S$, which carries a Quillen metric
according to the theorems of Quillen \cite{qui}, Zograf-Takhdajan
\cite{TaZo0}, and Bismut-Gillet-Soulé \cite{bgs}, once smooth metrics are
chosen on $K_{\cX/S}+ {\bf a}$. We show that the conical metrics on the
fibers induce a $C^\infty$ metric on the determinant line bundle, which
descends to the moduli space. As in the classical case, its curvature is
the generalized \wp form.

We also prove the formula for the curvature tensor of the \wp metric for
Riemann surfaces with conical singularities. In the classical case the
curvature was computed in \cite{Roy,F-T,Wo}. Our formula holds for the
case of weights $>1/2$, which is also the range, where Fenchel-Nielsen
coordinates exist. It includes also the case of orbifold singularities of
degree $m>2$.

Although hyperbolic conical metrics are well understood from the
standpoint of hyperbolic geometry, the dependence upon holomorphic
parameters poses essential difficulties. For this reason it was necessary
to introduce an ad-hoc definition of harmonic Beltrami differentials in
our previous paper \cite{S-T}, on which a \wp inner product could be
based. Our present results are valid with no restrictions on the weights,
in particular they include the interesting cases of weights between $1/2$
and $1$, which arise in the case of finite group quotients. Most results
are known for cusps i.e.\ punctures with zero cone angle, however our
approach seems to be only suitable for positive cone angles so that we
avoid mixed cases.

Acknowledgement. The first named author would like to thank Inkang Kim for
stimulating discussions. The authors would also like to thank the referee
for his or her helpful comments.

\section{Hyperbolic conical metrics} \label{conic}
Let $X$ be a compact Riemann surface with $n$ punctures $p_1, \ldots. p_n$, and
weights $0<a_j\leq 1$ for $j=1,\ldots,n$. We denote by ${\bf a}= \sum_j a_j
p_j$ the corresponding $\R$-divisor and by $(X,{\bf a})$ the weighted punctured
Riemann surface. We say that a hermitian metric of class $C^\infty$ on the
punctured Riemann surface, has a cone singularity of weight ${\bf a} $, if in a
holomorphic local coordinate system centered at $p_j$ the metric is of the form
$( \rho(z)/|z|^{2 a_j}) |dz|^2$ for $ 0<a_j <1$, whereas it is of the form
$(\rho(z)/ |z|^2 \log^2(|1/z|^{2}) )  |dz|^2$  if $a_j =1$. Here $\rho$ is
continuous at the puncture and positive. The cone angle is  $ 2 \pi (1- a_j)$,
including the complete case with angle zero.  Let ${K_X} $ be the canonical
divisor of $X$; the weighted punctured Riemann surface $(X, {\bf a})$ is called
\textit{stable}, if the the degree of the divisor $K_X +  { \bf a}$ is
positive. In this case, by a result of McOwen and Troyanov \cite{Mc-O},
\cite{Mc-O2}. \cite{T}, there exist a unique conical metric $g_{\bf{a}}$ on $X$
in the given conformal class, which has constant curvature $-1$ and presribed
cone angles. Moreover  $ {\rm Vol}(X,g_{\bf{a}})/\pi = \deg(K_X + { \bf a}) = -
\chi(X,\A)$. Where by definition $ \chi(X,\A) = \chi(X) - \sum a_j $ is the
Euler-Poincar\'e characteristic of the weighted punctured Riemann surface
$(X,\A)$.

At a non-complete conical puncture, we consider an emanating geodesic and
see that on a neighborhood of the puncture the hyperbolic metric is
isometric to a classical cone metric as obtained from the unit disk by
removing a sector and identifying the resulting edges. So a posteriori a
conical metric satisfies a somewhat stronger regularity condition than
predicted in terms of the partial differential equation for hyperbolicity.

\begin{remark}\label{re:grealana}
Let $(X,\A)$ be a weighted Riemann surface and $p_j$ a puncture with $0 < a_j <
1$ for all $1 \leq j \leq n$. Then there exist a local coordinate function $z$
near $p_j$ such that $\ga=(\rho(z)/|z|^{2a_j})|dz|^2$, where
$\rho(z)=\eta(|z|^{2(1-a_j)})$ for some positive, real analytic function
$\eta$.
\end{remark}

The dependence of the hyperbolic cone metrics on the weights is characterized
as follows.
\begin{proposition}\label{depweight}
Let  $ a_j(k) $ be an increasing sequence of weights with $\R$-divisors
${\bf a}(k)$ on $X$. Suppose that\/ $\deg(K_X + {\bf a }(k)) > 0$ for all
$ k \in \N $ and that $ a_j(k)  \rightarrow  a_j$, as $ k \rightarrow
\infty$. Then $g_{{\bf a}(k)} $ converges to $g_{ \bf a}$ uniformly on
compact sets away from the punctures. Moreover  the  sequence of functions
$ g_{{\bf a}(k)}/{g_{{ \bf a}}}$, converges to the constant function $1$
in $L^1(X, g_{ \bf a})$. \label{limit}
\end{proposition}

\begin{proof}
In Proposition 2.4 in \cite{S-T} we defined $\Psi_k = g_{{\bf a}(k)}/g_{{
\bf a} }$, then  $ 0 < \Psi_k  \leq \Psi_{k+1} \leq 1$ as we proved there,
and $- \log(\Psi_k)$ is a decreasing  sequence of subharmonic functions on
the complement of the punctures.  Therefore $-\log(\Psi_k)$ converges
pointwise to a subharmonic function $\delta \geq 0$ on the complement of
the punctures.  By Proposition 2.5 in \cite{S-T},  the function $ \delta$
is identically equal to $0$ in a neighborhood  of each puncture $p_j$ with
$a _j < 1$.  Moreover if $a _j < 1$ for all $j$ then $\delta \equiv 0$ and
the convergence is uniform on compact sets by Dini's lemma. (Observe that
the argument in the proof of Proposition 2.5 in \cite{S-T} is local).
Suppose that $a _{j_0} =1$ for some $j_0$,  and consider the functions $
\delta_k = - \log(\Psi_k) + (1- a(k)_{j_0}) \log(|z|^2)$, on  an open
neighborhood $\mathcal{U}_{j_0} $ of $p_{ j_0}$. By the local expression
of each function $\Psi_k$ near $p_{j_0}$, we have that the functions
$\delta_k$ are subharmonic and uniformly bounded from above, so each
function $\delta_k$ extends to a subharmonic function on
$\mathcal{U}_{j_0}$, moreover the function $ \delta' $ which is the upper
semi-continuous envelope of $\limsup \delta_k$ is also subharmonic on
$\mathcal{U}_{j_0}$, (cf.\ \cite{Kl}). Hence $\delta = \delta'$ on
$\mathcal{U}_{j_0} \backslash \{0 \}. $  In other words the function
$\delta$ extends to a subharmonic function on all of $X$, therefore
$\delta \equiv c$ is constant. By the dominated convergence theorem the
sequence $\Psi_k$ converges to $e^{-c}$ in $L^1(X, g_{ \bf a}  )$.  Since
${\rm Vol}({ g_{{\bf a}(k)}})$ converges to ${\rm Vol}({ g_{{\bf a} }})$,
we have $e^{-c} =1$.
\end{proof}

We consider the classical \tei space $\cT_{\gamma,n}$ of (marked) Riemann
surfaces of genus $\gamma$ with punctures $p_1,\ldots,p_n$.  We denote by
$\Pi:\cX_{\gamma,n} \to \cT_{\gamma,n}$ the universal family. The
punctures on the fibers $\cX_s =\Pi^{-1}(s)$ are given by $n$ holomorphic
sections $\sigma_1(s),\ldots, \sigma_n(s)$; $s \in \cT_{\gamma,n}$, where
for all $s$ the values are pairwise distinct. Constant weights $0< a_j
\leq 1$ are assigned to the $\sigma_j(s)$, and the corresponding real
divisors are denoted by ${ \bf a}(s) = \sum_{j=1}^n a_j \sigma_j(s)$. The
resulting family of weighted punctured surfaces is denoted by $ \Pi : {(
\mathcal{X}_{\gamma,n}, {\bf a})} \rightarrow \cT_{\gamma,n}$. We assume
that the fibers are stable and endowed with the hyperbolic conical metrics
$g_{\bf{a}}(s)$. The complete case of weights one is well-understood, and
since the essential arguments will be local, we may assume that for all
weights $0<a_j<1$ holds.

We will show that the conical hyperbolic metrics define new Kähler
structures on the \tei and moduli spaces of punctured Riemann surfaces
depending on the assigned weights.

For short we will write $\Pi:\cX \to S$ for any holomorphic family of
punctured Riemann surfaces over a complex manifold $S$ with holomorphic
sections $\sigma_i(s)$. Our arguments will be local with respect to the
base.

When considering the variation of conical metrics and defining the induced
hermitian structure on the \tei space, we may assume that $S = \{ s \in \C
: |s| < 1 \}$.

Denote by $X$ the central fiber ${\mathcal{X}}_0$. In order to introduce
Sobolev spaces, and to use the theory of elliptic equations depending upon
parameters \cite{A}, we need to fix a differentiable trivialization of the
family. Our method of choice is the following:

After shrinking $S$ if necessary, on neighborhoods of each  holomorphic section
$\sigma_j$ in $\mathcal{X}$ we take holomorphic coordinates $W_j \equiv {
\mathcal{U}}_j \times S = \{ (z,s) \}$ such that $ \sigma_j(s) \equiv 0$.
Assuming that these coordinates also exist on slightly larger neighborhoods we
can use a differentiable trivialization $ \Psi : \mathcal{X} \rightarrow X
\times S$, which is holomorphic on $W_j$ and respects the above coordinates.
The map $\Psi$ defines a differentiable lift
$$ V_0 = \frac{\partial}{\partial
s} +b_1(z,s) \frac{\partial}{\partial z} + b_2(z,s) \frac{\partial}{\partial
\overline{z}}
$$
of the vector field $\frac{\partial}{\partial s}$ on $S$, such that $V_0| W_j =
\frac{\partial}{\partial s}$. We introduce Sobolev spaces $H^p_k( { \mathcal{X
}}_s)$ defined with respect to the measure induces by a smooth family $g_0(s)$
of differentiable background metrics. We  identify $H^p_k( { \mathcal{X }}_s)$
with $H^p_k( X) $ by the above differentiable trivialization.

Set
\[
g_{ \bf a}  = e^u g_0
\]
where $ g_{ \bf a}(s) = g_{ \bf a}(s,z) |d z|^2 $ and   $ g_0(s) =  g_0 (s,z)
|d z|^2$ in local coordinates.  The functions $u$ carry the singularities.

Like in \cite{S-T}, section 4, for $1 \leq j \leq n$  we introduce a  function
$ \Psi_j(z,s)$ which is smooth on the complement of the punctures, and of the
form $\Psi_j  =  - \log(|z|^2 | )$ on  $\mathcal{U}_j$. (Here we use our
assumption that $\sigma_j(s) \equiv 0$.) Let us define
\[ w(z,s) = u - \sum_j a_j \Psi_j .\] Let
$ \Delta = \frac{1}{g_0} \frac{\partial}{\partial z \partial \overline{z}}$
denote the laplacian with respect to the smooth background metric $g_0$. Then
the equation for hyperbolicity reads
\begin{equation}
\Delta u - e^u = K_{g_0}  \label{main eq}
\end{equation}
where $K_{g_0}$ is the Ricci curvature of $g_0$, i.e.
$$ K_{g_0}(s,z) = -
\frac{1}{g_0(s,z)}\cdot \frac{\partial^2 \log(g_0)}{ \partial z
\partial \ol{z}}.
$$
Now equation (\ref{main eq}) reads:

\[
\Delta w - (e^{ \sum_i a_i \Psi_i}) e^w = K -\Delta(\sum_i a_i \Psi_i),
\]
and on $\mathcal{U}_j$ it is of the form
\[
\Delta w - e^{M(z)}  \frac{e^w}{|z|^{2a_j}} = K,
\]
where the function $M(z) = \sum_{i \neq j} a_i \Psi_i $  is smooth and bounded
on $\mathcal{U}_j$.

It follows that $w(s) \in H^p_2(\mathcal{X}_s) \ \mbox{for all} \  1 \leq p <
\min(1/a_j)$  (cf.\ \cite{Mc-O}), and by standard regularity theory the
solutions are of class $C^{\infty}$ on the complement of the punctures.

Our aim   is to show that the conical metrics depend  differentiably on the
parameters in a suitable sense. Given a family $(\mathcal{X}, {\bf a})
\rightarrow S$, we write the hyperbolic metrics as
$$
g_{ \bf a} = \exp( a_1 \Psi_1 + \ldots + a_n \Psi_n + w)\, g_0
$$
and fix a differentiable trivialization $ \mathcal{X} \rightarrow X \times S$
in the above sense.
\begin{theorem}\label{smooth}
Fix a real number $ 1 \leq p < \min( 1/{a_j})$. Then the assignment $ s
\mapsto w(s)$ defines a map $w : S \rightarrow H_2^p(X)$ which is of class
$C^{\infty}$, i.e.\ all higher derivatives of $w$ with respect to $V_0$
and $\overline{V_0}$ exist in $H^p_2(X)$ and depend in a $C^{\infty}$ way
on $s$. In particular, since $ H_2^p(X)  \subseteq  C^0(X)$, for any fixed
$z \in X$, the functions $s \mapsto w(z,s)$ is of class $ C^{\infty}$.
\end{theorem}

\begin{proof}
Since the argument is local, we may assume $n=1$ for simplicity. We define a
$C^1$ map $  \Phi : S \times H^p_2(X) \rightarrow L^p(X)$, by
\[
\Phi(s,w) = \Delta_{g_0(s)}(w) - e^{a \Psi(s)} e^w -K_{g_0(s)} +
a \Delta_{g_0(s)}(\Psi(s)).
\]
It is important to note that the given trivialization is holomorphic in a
neighborhood of the punctures and that  $\Psi(z,s) = - \log(|z|^2)$ does
not depend on $s$. Therefore the map $\Phi$ is of class $C^1$. We now
indicate how to compute $(D_1 \Phi)(s_0,w_0) \in L^p(X)$.  We have:
\begin{eqnarray*}
(D_1 \Phi)(s_0,w_0) &=& \frac{ - \partial \log g_0(s_0)}{ \partial s}
\Delta_{g_0(s_0)}(w_0) -a  \frac{  \partial \Psi (s_0)}{ \partial s}
e^{a \Psi(s_0)} e^{w_0}
\\
& &\qquad - \frac{\partial K_{g_0} (s_0)}{ \partial s} + \frac{\partial}{\partial s}
\left( \Delta_{g_0(s_0)}(\Psi(s,-)) \right).
\end{eqnarray*}
Note that the above function belongs to $L^p(X)$ since $ \Delta_{g(s_0)}(w_0) \in L^p(X)$ and
$ \frac{ \partial \Psi }{\partial s} =  \Delta_{g(s_0)}( \Psi ) \equiv 0$ near the puncture,
for all $ s \in S$. Moreover both of the functions
$ \left.\frac{  \partial \log g_0}{ \partial s}\right|_{s_0} $ and
$ \left.\frac{  \partial K_{g_0} }{ \partial s}\right|_{s_0} $ are bounded.
Now
\[
(D_2 \Phi)(s_0,w_0) (W) : H^p_2(X) \rightarrow L^p(X)
\]
is given by
\[
(D_2 \Phi)(s_0,w_0) (W) =   \Delta_{g_0(s_0)}(W)   - e^{a \Psi(s_0)} e^{w_0} W.
\]
Because of \cite[Lemma 2.1]{S-T}, the implicit function theorem is applicable.
Since all derivatives of $e^{a \Psi}$ with respect to $s$ and $ \bar{s}$ are in
$L^p(X)$, it is possible to repeat the argument, so that one can show the rest
of the statement.
\end{proof}
\begin{remark}\label{depweightrem}
The above methods can also be used to show that an analogous statement is true
for the dependence of conical metrics on the weights, provided these are less
than one. For ${\bf a}=\sum p_j$ we have the statement of
Proposition~\ref{depweight}.
\end{remark}

\section{The generalized Weil-Petersson metric} \label{WP}
The classical Weil-Petersson metric is defined as the $L^2$-inner product
of harmonic Beltrami differentials with respect to the hyperbolic metrics
on the fibers.

For reasons, which will become apparent later, we first introduce the \wp
metric on the cotangent space.

Let $(X, {\bf a})$ be a weighted punctured Riemann surface with ${\bf
a}=\sum a_jp_j$. We set $D = \sum p_j$ and denote by
$$
H^0(X,
{\Omega^2}_{(X, {\bf a})}) = H^0(X,{\Omega^2_X}(D))
$$ the space of
holomorphic quadratic differentials with at most simple poles at the
punctures, identified with the cotangent space of the corresponding \tei
space of punctured Riemann surfaces at the given point.

\begin{definition}
The Weil-Petersson inner product
\[
G^*_{WP,\bf a} \ \mbox{on  $H^0(X, {\Omega^2}_{(X, {\bf a})})$}
\]
is given by
\[
\langle \phi, \psi\rangle_{WP, {\bf a}} \  = \int_X \frac{\phi
\ol{\psi}}{g_{\bf a}^2} dA_{ \bf a} ,
\]
where $g_{ \bf a}$ is the hyperbolic conical metric, with surface element  $d
A_{\bf a}$.
\end{definition}
Observe that the above integrals are finite, because $0 \leq a_j \leq 1$ for
all $i$.

The \wp inner products depend continuously on the weights, if these are less
than one (cf.\ Remark~\ref{depweightrem}), and under the hypotheses of
Proposition~\ref{limit} we have the following statement.
\begin{corollary}
Let
\[
\phi \in H^0(X, {\Omega^2}_{(X, {\bf a})}),
\]
then
\[
\lim_k |\phi|^2_{WP, g_{ \bf a}(k)} =  |\phi|^2_{WP, g_{ \bf a} }.
\]
\end{corollary}

\begin{proof}
Fix a reference smooth metric $g_0$ on $X$. Then $ {| \phi|^2}/{g_{ {\bf
a}(k)}} $ is a   decreasing sequence of $g_0$ integrable  positive functions
converging to $ {| \phi|^2}/{g_{ {\bf a} }}.  $  \end{proof}

Observe that harmonicity of Beltrami differentials in the first place
means that a certain partial differential equation holds. In the case of
compact Riemann surfaces (and punctured surfaces equipped with complete
hyperbolic metrics) $L^2$-theory implies that any Beltrami differential
has a unique harmonic representative, which is the quotient of a conjugate
holomorphic quadratic differential by the metric tensor.

We use an ad hoc definition of the space of {\em harmonic Beltrami
differentials} for $(X, {\bf a})$ with respect to the hyperbolic conical
metric $g_{\bf a}$, which coincides with the usual definition in the
classical case of weights one. Let $X'=X\backslash \{p_1,\ldots,p_n\}$.

\begin{definition} \label{harmbelt}
Let $g_{ \bf a} = g_{ \bf a}(z) dz \overline{d z}$ be the hyperbolic conical
metric on $(X, {\bf a})$. If $ \phi = \phi(z) d z^2 \in H^0(X, {\Omega^2}_{(X,
{\bf a})})$ is a quadratic holomorphic differential, we call the Beltrami
differential
\[
\mu=\mu(z) \frac{\partial}{\partial z} \overline{d z} =
\frac{\overline{\phi(z)}}{g_{\bf a}(z)} \frac{\partial}{\partial
z}\overline{d z},
\]
on $X'$ {\em harmonic} on $(X, {\bf a})$ and denote the vector space of all
such differentials by $H^1(X, {\bf a})$.
\end{definition}

\begin{proposition}
For $ 0 < a_j <1$ the space of harmonic Beltrami differentials $H^1(X, {\bf
a})$ on $(X,{\bf a})$ can be identified with the cohomology
$H^1(X,\Theta_X(-D))$, where $\Theta_X$ is the sheaf of holomorphic vector
fields on $X$ and $D = \sum_j p_j$.
\end{proposition}
\begin{proof}
It is sufficient to verify that a duality
\[
\Phi :H^0(X, {\Omega^2}_{(X, {\bf a})}) \times H^1(X, {\bf a})
\rightarrow \C.
\]
is defined by
\[
\Phi\left( \phi(z) dz^2, \mu(z) \frac{\partial}{\partial z}
\overline{d z}\right) = \int_X \phi(z) \mu(z) dz d \ol{z}.
\]
\end{proof}
The Weil-Petersson metric on the cotangent space to ${\mathcal T}_{\gamma,n}$
together with the above duality defines a \wp metric $ G_{WP, {\bf a}} $ on the
tangent space identified with $H^1(X, {\bf a})$.

Let $\mu_1$, and $\mu_2$ in $H^1(X, {\bf a})$, then
\[
\langle \mu_1, \mu_2\rangle_{WP, {\bf a}} \  = \int_X \mu_1
\overline{ \mu_2}d A_{ \bf a}.
\]
(cf.\ \cite[Lemma 3.4.]{S-T}).

If $1/2 \leq  a_j \leq 1$ then the Fenchel-Nielsen coordates can be
defined, \cite{ T-W-Z}, it is shown in \cite{A-Sch} that in this case the
Fenchel-Nielsen symplectic form coincides with the Weil-Pertersson
K\"{a}h\-ler form. The generalized \wp metric  can be defined  on the \tei
space ${\mathcal{T}}_{\gamma,n}$ of surfaces of genus $\gamma$ with $n$
punctures. From Proposition 2.4  of \cite{S-T} we know that if $ { \bf a}
\leq { \bf b} $ then $g_{\bf a} \leq g_{\bf b }$, hence $ G^*_{WP, {\bf
b}} \leq G^*_{WP, {\bf a}}$, and for the metrics on the dual spaces we
have $G_{WP, {\bf a}} \leq G_{WP, {\bf b}}. $ Therefore, if\/ $ {\bf a }
\leq { \bf b} $, the identity map from $({\mathcal{T}}_{\gamma,n}, G_{WP,
{\bf b}}) $ to $({\mathcal{T}}_{\gamma,n}, G_{WP, {\bf a}}) $ is distance
decreasing.

Since the conical metrics are intrinsically defined on the fibers, the
classical mapping class group $\Gamma_{\gamma,n}$ acts on \tei spaces as a
group of isometries for both the classical and the generalized \wp
metrics, hence also the generalized \wp metric descends to $
\mathcal{M}_{\gamma, n}. $ Let us define $ \overline{
\mathcal{M}}_{\gamma, { \bf a}} $ as the completion of  the moduli space $
\mathcal{M}_{\gamma, n } $ with respect to the distance defined by the
generalized \ metric. Therefore the identity map descends to a distance
decreasing map of the moduli spaces, and such a map extends to a
continuous map
\[
j_{ { \bf b}, { \bf a} } :   \overline{ \mathcal{M}}_{ \gamma,{ \bf b}}   \rightarrow
\overline{ \mathcal{M}}_{ \gamma,{ \bf a}}.
\]
Moreover let $ { \bf b } = ( {\bf b'},{\bf b''})$, and $  {\bf b^*} =
({\bf b'},0)$  where $ { \bf b' } \in [0,1]^m$. Denote by $ F : {\mathcal{
T}}_{\gamma,n} \rightarrow {\mathcal{T}}_{\gamma,m}$ the holomorphic map,
which forgets the punctures ${\bf b}''$.  Then by \cite[Theorem 3.5]{S-T}
$G_{WP, {\bf b^*}} $ coincides with the degenerate metric $ F^*( G_{WP,
{\bf b'}})$. The map
\[
F : ( {\mathcal{ M}}_{\gamma,n},  F^*( G_{WP, {\bf b'}}))
\rightarrow ({\mathcal{M}}_{\gamma,m},  G_{WP, {\bf b'}})
\]
is also obviously (psudo)distance decreasing, and since  $ { \bf b } \geq
{ \bf b^*}$, so is the map $ F = F \circ id : ( {\mathcal{ M}}_{\gamma,n},
G_{WP, {\bf b}})) \rightarrow ({\mathcal{M}}_{\gamma,m}, G_{WP, {\bf
b'}})$.

Therefore  we also have the continuous  map forgetting punctures
\[
F_{ { \bf b}, { \bf b'} } :   \overline{ \mathcal{M}}_{ \gamma,{ \bf b}}   \rightarrow
\overline{ \mathcal{M}}_{ \gamma,{ \bf b'}}.
\]

\begin{corollary}

The space $ \overline{ \mathcal{M}}_{\gamma, { \bf a}} $ is a compactification
of the  moduli space $  \mathcal{M}_{\gamma, n}$.  In particular the
generalized \wp metric is not complete.\end{corollary}

\begin{proof}
The usual Deligne-Mumford compactification of ${\mathcal{ M}}_{\gamma,n} $ is
the quotient by the mapping class group of the \wp metric completion of \tei
space, see for example \cite{Ma},  \cite{Wo-2}, hence it is the completion of
${\mathcal{ M}}_{\gamma,n}. $ Therefore if $ { \bf 1 } = (1,\ldots,1)$, then
$j_{ {\bf 1}, {\bf a}}( \overline{ \mathcal{M}}_{ \gamma,{ \bf 1}})  \subseteq
\overline{ \mathcal{M}}_{ \gamma,{ \bf a}}$, is compact and dense, so that  the
map $j_{ {\bf 1}, {\bf a}}$ is onto and $ \overline{ \mathcal{M}}_{ \gamma,{
\bf a}}$ is compact. \end{proof}

\section{The Kodaira-Spencer map and conical metrics}
First, we briefly describe the close relationship of variations of
hyperbolic metrics and harmonic Beltrami differentials in the classical
case of holomorphic families of compact Riemann manifolds (cf.\ also
\cite{Sch}).

Let $f:\cX \to S$ be such a family. Let $s_0\in S$ be a distinguished
point and $X=f^{-1}(s_0)$ its fiber. The map induces a short exact
sequence involving the sheaf $\cT_{\cX/S}$ of holomorphic vector fields in
fiber direction, the sheaf of holomorphic vector fields $\cT_{\cX}$ on the
total space and the corresponding pull-back:
$$
0 \to \cT_{\cX/S} \to \cT_\cX \to f^*\cT_S \to 0.
$$
The connecting homomorphism
$$
\rho : T_{s_0} \to H^1(X,\cT_X)
$$
is the Kodaira-Spencer map, which in fact assigns to a tangent vector the
cohomology class of the corresponding Beltrami differential.

In terms of Dolbeault cohomology, this map can be described as follows:
Let $\partial/\partial s$ stand for a tangent vector on the base at $s_0$.
Let $V$ be any differentiable lift of the tangent vector to the total
space $\cX$ (along the fiber $X$).

\begin{proposition}
The restriction $\ol\partial V|X$ is $\ol\partial$-closed and represents
$\rho(\partial/\partial s|_{s_0})$.
\end{proposition}

Now the fibers $\cX_s$ of the family are equipped with the hyperbolic
metrics $g(z,s)|dz|^2$, which depend in a differentiable way on the
parameter $s$. The collection of these metrics is considered a relative
volume form on the total space $\cX$, its dual is a hermitian metric on
the relative canonical bundle $\cK_{\cX/S}$. Let
$$
\omega_\cX = \frac{\sqrt{-1}}{2} \partial_{\mathcal{X}} \ol{
\partial}_{\mathcal{X}} \log(g(z,s))
$$
be its curvature form.
\begin{lemma}
The restrictions of $\omega_\cX$ to the fibers $\cX_s$ equal the Kähler
forms $\omega_{\cX_s}= \frac{\sqrt{-1}}{2} g(z,s) dz \wedge \ol{dz}$.
\end{lemma}
In particular the real $(1,1)$-form $\omega_\cX$ is positive definite
along the fibers. So the {\it horizontal lift} $V_{hor}$ of
$\partial/\partial s$, which by definition consists of tangent vectors
that are perpendicular to the fibers and project to the given tangent
vector, is well-defined:
\begin{lemma}\label{le:hori}
$$
V_{hor}= \left.\frac{\partial}{\partial s}\right|_{s_0} + a^z \frac{\partial}{\partial z}
$$
with
$$
a^z = - \frac{1}{g}\frac{\partial^2 \log g(z,s_0)}{\partial s \partial \ol z}.
$$
\end{lemma}
The Lemma follows immediately from the computation of the inner product of
$V_{hor}$ and $\pt/\pt z$ with respect to $\omega_\cX$ .

So far general theory implies the following:
\begin{proposition}\label{pr:harmbel}
The harmonic Beltrami differential corresponding to the tangent vector
$\partial/\partial s|_{s_0}$ is induced by the horizontal lift. It equals
$$
\mu= \mu(z)\frac{\partial}{ \partial z} \  d
\ol{z} = \frac{\pt a^z}{\pt\ol{z}} \frac{\partial}{ \partial z} \  d
\ol{z}= - \frac{\partial}{\partial \ol z} \left( \frac{1}{g}\frac{\partial^2
\log g(z,s_0)}{\partial s \partial \ol z} \right) \frac{\partial}{ \partial z} \  d
\ol{z}.
$$
\end{proposition}
In fact, a straightforward verification shows that $g(z,s_0) \ol{\mu(z)}$
is a holomorphic quadratic differential, i.e.\ $\mu$ is {\it harmonic}
with respect to the hyperbolic metric on $X$.

Now let $( \mathcal{X},  { \bf a }) \to S$ be a holomorphic family of
weighted Riemann surfaces with $ 0 < a_j < 1$, and with central fiber $X =
\mathcal{X}_{s_0}$, $s_0 \in S$. This section is concerned with how to
recover the Kodaira-Spencer map $ \rho : T_{s_0}(S) \rightarrow H^1(X,  {
\bf a })$ from the family of conical hyperbolic metrics $ g_{ \bf a}$.

In the case of conical hyperbolic metrics we define the Beltrami
differential  given by
\begin{equation}\label{Belt}
\mu_{ {\bf a}}\left(\frac{\partial}{\partial s}\right) =-
{\frac{\partial}{\partial \ol{z}} \left(\frac{1}{g_{ {\bf a}}}
\frac{ \partial^2 \log g_{ {\bf a}}}{\partial \ol{z} \
\partial s} \right)} \frac{\partial}{ \partial z} \  d
\ol{z}.
\end{equation}
and the quadratic differential $\phi_{ {\bf a}}(\frac{\partial}{\partial s}) =
g_{ {\bf a}} \overline{ \mu_{ {\bf a}}}(\frac{\partial}{\partial s}) $.

In order to prove that the above Beltrami differential $ \mu_{ {\bf
a}}\left(\frac{\partial}{\partial s}\right)  $ is harmonic in the sense of
Definition~\ref{harmbelt} it is sufficient to show the following:

\begin{lemma}\label{le:phiL1}
 $\phi_{ {\bf a}}(\frac{\partial}{\partial s})  $ is in $L^1(X)$. \label{L1}
\end{lemma}

\begin{proof}
Again we use the special coordinates for the family near the punctures. For
simplicity we assume $n=1$ and set $ 0 < a = a_1 < 1,  \ g_a = g_{ \bf a }$. We
have
\[
\phi_{\bf a}\left(\frac{\partial}{\partial s}\right) = \frac{
\partial \log g_{a}}{\partial z } \cdot  {\frac{
\partial^2 \log g_{a}}{\partial z \ \partial \ol{s}}
 } - { \frac{ \partial^3 \log g_{a}}{\partial z^2 \
\partial \ol{s}} }.
\]
Moreover in local coordinates the following equation holds:
\begin{equation}\label{L^1}
\log(g_a) = \log(g_0) + w -a \log(|z|^2).
\end{equation}

Now  by Theorem \ref{smooth} we have for $1 \leq p < \frac{1}{a} $ that
$$
\frac{ \partial w}{\partial z},\frac{ \partial^2  w}{ \partial \bar{s}\partial z} \in H_1^p(\mathcal{U}_1)
$$
whereas
$$
\frac{ \partial^3  w}{ \partial \bar{s} \partial z^2} \in L^p(\mathcal{U}_1).
$$
Therefore by equation \eqref{L^1}
$$
{ \frac{ \partial^3 \log g_{\bf
a}}{\partial z^2 \ \partial \ol{s}} } \in L^1(\mathcal{U}_1).
$$

Moreover  $ 1/z \in L^q(\mathcal{U}_1) $ therefore
$$
\frac{\partial \log g_{\bf a}}{\partial z }  \in L^q(\mathcal{U} _1)
\text{\quad for \quad} 1 \leq q < 2\;.
$$

By the Sobolev embedding theorem $ H_1^p(\mathcal{U}_1) \subseteq L^{h}(
\mathcal{U}_1)$ for all $h < p'$, where $p' = \frac{2p}{2-p} $ for $ 1 \leq p <
2$ and $p' = \infty$ for  $p \geq 2$.

If follows that
\begin{eqnarray*}{\frac{\partial^2 \log g_{\bf a}}{\partial z \ \partial \ol{s}} }  &\in& L^h(\mathcal{U}_1)
\text{\; for \;} 1 \leq h < \infty \text{\; if \;} 0 < a \leq 1/2  \\  \text{and}\hfill &&\\
{\frac{\partial^2 \log g_{\bf a}}{\partial z \ \partial \ol{s}} }  &\in& L^h(\mathcal{U}_1) \text{\; for \;}
1 \leq h <   \frac{1}{a -1/2}  > 2  \text{\; if \;} 1/2 < a < 1 .
\end{eqnarray*}
Hence for $ 0 < a < 1 $
$$
\frac{ \partial \log g_{\bf a}}{\partial z } \cdot
{\frac{ \partial^2 \log g_{\bf a}}{\partial z \ \partial \ol{s}}  }  \in
L^1(\mathcal{U}_1).
$$
\end{proof}

So far we only showed that, on one hand $H^1(X, { \bf a})$ is the space of
infinitesimal   deformations, and that on the other hand, the variation of
hyperbolic conical metrics gives rise to element of this space according to
(\ref{Belt}). If this assignment is injective for effective families, then we
recovered the Kodaira-Spencer map.

\begin{theorem}\label{th:ks}
The Kodaira-Spencer map $ \rho: T_{s_0} S \rightarrow H^1(X, { \bf a})$ is
given by
\[
\rho\left( \frac{\partial}{\partial s}\right) = \mu_{ {\bf a}}\left(\frac{\partial}{\partial s}\right) = - \left.{ \frac{ \partial}{ \partial \ol{z}} \left(
\frac{1}{g_{\bf a}} \frac{ \partial^2 \log(g_{\bf a}(z,s))}{\partial \ol{z} \
\partial s} \right) }\right|_{s=s_0} \ \frac{\partial}{ \partial z} \  d
\ol{z}
\]
where $   \frac{\partial}{\partial s} $  stands for a tangent vector.
\end{theorem}

\begin{proof}
We may assume that $S$ is a disk and that we only have one puncture. If $0
< a < 1/2$ the proof of the Theorem is given in \cite[Theorem 5.4]{S-T},
so we suppose $1/2 \leq a < 1$. Let $ \mu_{ {\bf
a}}\left(\frac{\partial}{\partial s}\right) \equiv 0$. Then the locally
defined quantity $\left.\frac{1}{g_a} \frac{
\partial^2 \log(g_a(z,s))}{\partial \ol{z} \ \partial s} \right|_{s=s_0}
$ is holomorphic outside the punctures, and the vector field
\[
W_{s_0} = \frac{\partial}{\partial s}+\gamma(z) \frac{\partial}{\partial z} = \frac{\partial}{\partial s}
- \left(\left.\frac{1}{g_{{ a}}} \frac{\partial^2\log(g_{ { a}}(z,s))}{\partial
\ol{z} \ \partial s}  \right|_{s= s_0}\right) \frac{\partial}{ \partial z}
\]
is a lift of the tangent vector
$ \frac{\partial}{\partial s} $ which is holomorphic outside the punctures. We
know from the proof of Lemma  \ref{L1} that $\left.\frac{ \partial^2 \log(g_{
a}(z,s))}{\partial \ol{z} \ \partial s} \right|_{s=s_0}$ is in $H_1^p(
\mathcal{U}_1)  \subseteq L^2( \mathcal{U}_1)$ for  some $ p >1$. Since
$\frac{1}{g_a}$ is bounded, the function $\left. \frac{1}{g_a} \frac{
\partial^2 \log(g_{ a}(z,s))}{\partial \ol{z} \ \partial s}   \right|_{s=s_0}
$ is also in $ L^2( \mathcal{U}_1)$, hence  the vector field is
holomorphic on the compact surface. So the holomorphic structure of the
corresponding compact Riemann surfaces is infinitesimally constant.
However, the puncture need not be kept fixed. Given the choice of local
coordinates, we need to show that the vector field $ W_{s_0}$ equals  $
\frac{\partial}{\partial s} $ at $z=0$.  We already observed that $
\left.\frac{ \partial^2 \log(g_ a(z,s))}{\partial \ol{z} \ \partial s}
\right|_{s=s_0} = \frac{\rho(z)}{|z|^{2a}} \gamma(z) $ is in
$L^2(\mathcal{U}_1)$, however for $1/2 \leq a < 1$ the function $
\frac{1}{|z|^{2a}}$ is not in $L^2( \mathcal{U}_1)$, hence
$\gamma(s_0)=0.$
\end{proof}

\section{Horizontal lifts of tangent vectors}\label{hor}
Let $ f : ( \mathcal{X},  { \bf a }) \rightarrow S$ be the universal
holomorphic family of weighted Riemann surfaces over the \tei space, or
for computational purposes, a family over the disk. Observe that like in
the classical case the family of conical metrics will give rise to a
$C^\infty$ closed, real $(1,1)$-form
$$
\omega_{\mathcal{X}} =  \frac{\sqrt{-1}}{2} \partial_{\mathcal{X}} \ol{
\partial}_{\mathcal{X}} \log(g_{\bf a})
$$
on the complement of the punctures, which is positive, when restricted to
the fibers.

Assume that  $ 1 < a_j < 1$ for $ 1 \leq j \leq n$. Let $S = \{ s \in \C ;
|s| < 1 \}$ and denote by $X =\mathcal{X}_0$ the central fiber. As in
Section \ref{conic} we use a differentiable trivialization of the family
so that the Sobolev spaces of the fibers can be identified.

We will denote the coefficients of $\omega_\cX$ by
\begin{eqnarray}
{g_{\bf a}}_{s \ol{s}}&=& \frac{\pt^2 \log g_{\bf a}(z,s)}{\pt s \ol{\pt s}} \label{eq:gzz} \\
{g_{\bf a}}_{s \ol{z}}&=& \frac{\pt^2 \log g_{\bf a}(z,s)}{\pt s \ol{\pt z}} \label{eq:gsz} \\
{g_{\bf a}}_{z \ol{s}}&=& \frac{\pt^2 \log g_{\bf a}(z,s)}{\pt z \ol{\pt s}} \label{eq:gzs} \\
{g_{\bf a}}_{z \ol{z}}&=& \frac{\pt^2 \log g_{\bf a}(z,s)}{\pt z \ol{\pt z}} \label{eq:gss} .
\end{eqnarray}
As pointed out above, hyperbolicity translates into
\begin{equation}\label{eq:hy}
{g_{\bf a}}_{z
\ol{z}}= g_{\bf a}.
\end{equation}
Like in Lemma~\ref{le:hori} we have that the horizontal lift of $\pt/\pt
s$ is given by
$$
V= (\pt/\pt s) + a^z(z)(\pt/\pt z)
$$
with
\begin{equation}\label{eq:as}
a^z=\frac{-1}{g_{\bf a}}   {g_{\bf a}}_{s \ol{z}} .
\end{equation}
The function
\begin{equation}
\chi =  {g_{\bf a}}_{s\ol{s}}  - \frac{1}{g_{\bf a}} {g_{\bf a}}_{s \ol{z}}  \ {g_{\bf a}}_{z \ol{s}}
={g_{\bf a}}_{s\ol{s}}- {g_{\bf a}}a^z(z) \ol{a^z(z)}
\end{equation}
has various geometric meanings:
\begin{proposition}
Let $\mu_{\bf a} \in  H^1(\mathcal{X}_{s_0}, {\bf a})$ be the harmonic
Beltrami differential according to \eqref{Belt}. Then
\begin{eqnarray}
  \chi &=& \|V\|^2_{\omega_\cX}\\
\omega_{\mathcal{X}}^2 &=& \left(\frac{\sqrt{-1}}{2}\right)^2  \chi(z,s) g_{\bf a}(z,s)d z \wedge d \bar{z} \wedge
d s  \wedge  d \bar{s}.  \label{omega}\\
| \mu_{\bf a}|^2    &=&\left( - \Delta_{g_{\bf a}} + id \right) \chi \label{laplace}
\end{eqnarray}
\end{proposition}
\begin{proof}
For simplicity we will drop the index $\bf a$ and we set $\pt_s = \pt/\pt
s$ and $\pt_z = \pt/\pt z$ etc. The first claim follows from
$$
\|V\|^2_{\omega_\cX} = \langle \pt_s + a^z \pt_z, \pt_s + a^z\pt_z\rangle = g_{s\ol s} + a^z g_{z\ol s} +
\ol{ a^z} g_{s\ol z} + a^z \ol{ a^z} g_{z \ol z}
$$
by \eqref{eq:as} and \eqref{eq:hy}. Equation \eqref{omega} follows from
$$
\chi \cdot g = \chi \cdot g_{z \ol z} =
\det
\begin{pmatrix}
 g_{s\ol s} & g_{s\ol z} \\
 g_{\ol z s} & g_{z \ol z}
\end{pmatrix}.
$$
The proof of \eqref{laplace} will require some preparations.
\end{proof}
In order to compute integrals over the fibers involving certain tensors,
we will use covariant differentiation with respect to the hyperbolic
metrics on the fibers and use the semi-colon notation. For derivatives in
$s$-direction we will use the flat connection.

First, we note that
\begin{gather*}
g^2 \cdot g_{s\ol{s}} =g^2 \cdot (\log g)_{;s\ol s}=g\cdot {g_{;s\ol s}} -
{g_{;s}g_{;\ol s}} =  g\cdot{g_{;s\ol s}} - {g_{z\ol z;s}g_{z\ol z; \ol
s}} \\= g\cdot {g_{;s\ol s}} - {g_{s\ol z;z}}{g_{z \ol s;\ol z}} =
 g\cdot g_{;s\ol s} - {g^2}\cdot a^z_{\; ;z} \ol{a^z}_{\; ;\ol z}
\end{gather*}
i.e.\
$$
\frac{1}{g}g_{;s\ol s}=  g_{s \ol s} + a^z_{\; ;z} \ol{a^z}_{\; ;\ol z}.
$$
We combine this with
$$
g_{s\ol s; z\ol z}= (\log g)_{;s\ol s z \ol z}=(\log g)_{;z\ol z s \ol s}= g_{;s \ol s}
$$
and get
\begin{gather*}
\Delta_g(\chi) = \frac{1}{g} (g_{s \ol s} - g \cdot a^z \ol{a^z})_{;z \ol
z} = \frac{1}{g} g_{; s \ol s} -(a^z \ol{a^z})_{;z \ol z}\\
= g_{s \ol s} - a^z_{\; ;\ol z} \ol{a^z}_{\; ; z} - a^z_{\; z \ol z }
\ol{a^z} - a^z \ol{a^z}_{\; ;z \ol z }.
\end{gather*}
We know that
$$
\mu(z) = a^z_{\; ;\ol z},
$$
hence
$$
\ol{a^z}_{\; ;z \ol z }=\ol{\mu(z)}_{;\ol z} = \left(\frac{\varphi(z)}{g}\right)_{;\ol z} = 0,
$$
where $\varphi$ is some holomorphic quadratic differential. Furthermore in
terms of the curvature tensor $R^z_{\; z z \ol z}$ and Ricci tensor
$R_{z\ol z} = - g$ resp.
$$
a^z_{\; z \ol z } =  a^z_{\; \ol z z } + a^z R^z_{\; z z \ol z} = \ol{\mu}_{; z} + a^z (-R_{z \ol z})
= g \cdot a^z.
$$
So
$$
\Delta_g(\chi) = \chi - |\mu|^2
$$
which ends the proof of the Proposition.

The equations are so far established on the complement of the punctures.

\begin{lemma}\label{Sob}
Let $h_0 = \min_j( \frac{1}{1-a_j})$ and \; $q_0 =  \min \left(
\min_j(\frac{1}{ a_j}), \min_j(\frac{1}{1-a_j}) \right)$. Then
\begin{itemize}
\item[(i)] $\frac{ |\mu|^2 g_{\bf a}}{g_0} \in L^h(\mathcal{X}_{s_0})$
    for $1 \leq h < h_0$.
\item[(ii)] $ \chi  \in H_2^q(\mathcal{X}_{s_0})$ for $1 \leq q <
    q_0$.
\item[(iii)] The functions $ s \mapsto \frac{ |\mu|^2 g_{\bf a}}{g_0}
    \in L^h(\mathcal{X}_{s}) \equiv  L^h(X) $ and $ s \mapsto \chi \in
    H_2^q(\mathcal{X}_{s}) \equiv H_2^q(X) $ are both of class
    $C^\infty$.
\item[(iv)] For the coefficient of the harmonic Beltrami differential
    $\mu(z) \in H^p_1$ for $p<h_0$ holds.
\end{itemize}
\end{lemma}

\begin{proof}
The expression $ \ |\mu|^2 \frac{g_{\bf a}}{g_0} $ in local coordinates
near the puncture $p_j$  behaves like $ \frac{1}{|z|^{2(1-a_j)}}$ because
of Lemma~\ref{le:phiL1}, hence (i) follows. Now we write equation
(\ref{laplace}) as
\[
- \Delta_{g_{0}} \chi + \frac{\ga}{g_0} \chi =  \frac{\ga}{g_0} | \mu|^2.
\]
However, near the puncture $p_j$,  the function  $ \frac{\ga}{g_0}$ is in
$L^p$ for $1 \leq   p < \frac{1}{a_j}$, so by \cite[Lemma 2.1]{S-T}
together with (i), the claim (ii) follows. To prove (iii) we apply Theorem
\ref{smooth} together with the smooth dependence on parameters of the
solution of elliptic equations. In order to see (iv), we express $\mu$ in
terms of a quadratic holomorphic differential and apply
Remark~\ref{re:grealana}.
\end{proof}

\begin{proposition}
For every point $s_0 \in S$, we have:
\[ \left\|\left.{ \frac{\partial  }{\partial s}}\right|_{s_0}\right\|^2_{W\!P,\bf a} =
\int_{\mathcal{X}_{s_0}}  \chi \; dA_{\ga}.
\]
\label{WPint}
\end{proposition}
\begin{proof}
We have \[ \int_X \Delta_{\ga} \chi \; dA_{\ga} = \sqrt{-1} \int_X
\partial \bar{\partial} \chi = 0 \] because $ \chi \in H^q_2(X)$ for some $
q > 1$,  $X$ is compact and the space of smooth functions is dense in
$H^q_2(X)$.  Now by equation (\ref{laplace})
\[
\int_X | \mu|^2 dA_{\ga} = \int_X \chi\; dA_{\ga}.
\]
\end{proof}
Assume now that $S$ is arbitrary and $f:\cX \to S$ a holomorphic family of
weighted punctured Riemann surfaces. We denote by $\omega^{WP}_S$ the real
$(1,1)$-form, which is determined by the \wp inner product of tangent
vectors on $S$: Given a tangent vector $u\in T_{S,s_0}$ we denote by
$\rho_{S,s_0}(u)= \mu_{\bf a}(u) \in H^1(X,\bf a)$ the corresponding
harmonic Beltrami differential in the sense of Theorem~\ref{th:ks}.

At this point, we introduce the notion of {\em fiber integrals} of
differential forms for a holomorphic family $f: \cX \to S$ of compact
complex manifolds of dimension $n$ say. Let $\eta$ be a differential form
of a certain degree $(k+n,k+n)$. Let
$$
\xymatrix{X \times S \ar[r]^\phi \ar[dr]_{pr} & \cX \ar[d]^f \\ & S}
$$
be a differentiable trivialization.  Then
$$
\int_{\cX/S} \eta := \int_{X \times S / S} \phi^* \eta
$$
denotes a differential form of degree $(k,k)$, where the latter integral
is defined in terms of the components of $\phi^* \eta$ which have total
degree $2n$ in fiber direction and degree $2k$ in $S$-direction. The
exterior derivative of a fiber integral can be computed in different ways.
Primarily
$$
d \left( \int_{\cX/S} \eta\right) = \int_{\cX/S}  d \eta.
$$
The latter integral can be evaluated in terms of $\phi$. Since a
differentiable trivialization determines a lift $v$ of tangent vectors
$\pt/\pt x$ of the base, any partial derivative
$$
\frac{\pt}{\pt x} \int_{\cX/S} \eta = \int_{\cX/S} L_v(\eta)
$$
where $L_v$ denotes the Lie derivative of the differential form $\eta$
with respect to $v$. On can verify that this is also true for
differentiable lifts of complex tangent vectors, which need not arise from
differentiable trivializations.

Then
\begin{theorem}\label{th:kaeh}
The fiber integral
\[ \int_{ \mathcal{X}/S } \omega^2_{\mathcal{X}} = \omega^{WP}_S \] equals
the \wp form.
\end{theorem}

\begin{proof}
Let $\alpha: \wt S \to S$ a holomorphic map of complex manifolds. We
consider $\wt \cX = \cX \times_S \wt S$ and the cartesian diagram
$$
\xymatrix{
\wt\cX \ar[r]^{\wt \alpha} \ar[d]_{\wt f} & \cX \ar[d]^f\\
\wt S \ar[r]^\alpha& S
}
$$
Since the hyperbolic metrics on the fibers $\wt \cX_t$ are just  the
hyperbolic metrics on  the $\cX_{\alpha(t)}$, $t \in \wt S$, the relative
volume form on $\wt \cX \to \wt S$ equals $\wt \alpha^* g$ where $g$
denotes the relative volume form for $\cX \to S$. This implies
$$
\wt \alpha^* \omega_\cX = \wt \alpha^*( \ii \pt\ol\pt \log g) =
\ii \pt\ol\pt \log \wt\alpha^*g = \omega_{\wt \cX}.
$$
Hence the integral in the above Theorem commutes with base change, in
particular with the restriction to local analytic curves.

On the other hand, the \wp Hermitian product i.e.\  the evaluation of
$\omega^{WP}$ at tangent vectors commutes with base change: For $v \in
T_{\wt S, t_0}$ we have $\rho_{\wt S,t_0}(v) = \rho_{S,
\alpha(t_0)}(\alpha_*(v))$. Hence
\begin{gather*}
\omega^{WP}_{\wt S}(v,w)= \langle \rho_{\wt S,t_0}(v),\rho_{\wt
S,t_0}(w)\rangle_{W\!P, \bf a } \hspace{5cm} \\= \langle \rho_{S,
\alpha(t_0)}(\alpha_*(v))  , \rho_{S, \alpha(t_0)}(\alpha_*(w))
\rangle_{W\!P, \bf a } = \omega^{WP}_{S}(\alpha_*(v),\alpha_*(w)),
\end{gather*}
hence
$$
\alpha^*\omega^{WP}_{S}=\omega^{WP}_{\wt S}.
$$

Since both $\omega_{\mathcal{X}}$ and $ \omega^{WP}$ are defined in a
functorial way, it is sufficient to check the case $\dim_{\C}S =1$, which
follows from Proposition \ref{WPint} and Formula (\ref{omega}). \end{proof}

\begin{theorem}\label{th:fibint}
The \wp form is of class $C^\infty$ and $d$-closed on the base of any
holomorphic family. In particular, on the Teichmüller space, $\omega^{WP}$
is a Kähler form.
\end{theorem}

\begin{proof}
At this point we introduce holomorphic coordinates $s^i$; $i=1,\ldots,N$
on the base space $S$. We consider the horizontal lifts $V_i$ on $\cX$ and
their inner product with respect to $\omega_\cX$
$$
\chi_{i\ol\jmath}= \langle V_i,V_j\rangle
$$
Furthermore
\begin{equation}\label{eq:laplace}
(\Delta_{g_{\bf a}} -id) \chi_{i\ol\jmath}= \mu_i \mu_{\ol \jmath}.
\end{equation}
The relevant term for the fiber integral of $\omega_\cX^2$ is
$$
\ii \chi_{i\ol\jmath}\; g_{\bf a}\; dA \; ds^i\wedge
ds^{\ol\jmath}.
$$

In order to show the Theorem need to prove that
$$
d \int_{\cX/S} \omega_\cX^2 =\int_{\cX/S} d( \omega_\cX^2),
$$
The map $S \to L^p$ , $p$ as above, which sends $s$ to $\chi_{i\ol\jmath}
g_{\bf a}/g_0$ is of class $C^\infty$, because of Theorem~\ref{smooth} and
Lemma~\ref{Sob}. So we apply a differerentiable local trivialization of
the family. Then
$$
\frac{\pt}{\pt s^k}\int_{X} \chi_{i\ol\jmath} g_{\bf a} dA =
\int_{X} F_{i\ol\jmath k} g_{\bf a} dA
$$
for some $F_{i\ol\jmath k} \in L^p(X)$. Since $L^p$-convergence of a
sequence implies pointwise convergence of a subsequence almost everywhere,
the function $F_{i\ol\jmath k}$ has to be the derivative of the integrand
outside a set of measure zero. This argument shows that exterior
derivative on $S$ of the differential form given by the fiber integral of
$\omega_\cX^2$ equals the fiber integral of the exterior derivative of
$\omega_\cX^2$ on the total space $\cX$. The latter form $d(\omega_\cX^2)$
is in $L^p$ and equal to zero outside a set of measure zero, so the
integral is identically zero.
\end{proof}

\section{Determinant line bundles an Quillen metrics in the conical case}\label{quillen}
Let $f:(\cX, \bf a) \to S$ be any holomorphic family of weighted punctured
Riemann surfaces equipped with the family $g_{\bf a}$ of conical metrics,
in particular $f$ may denote the universal such family. In this section we
consider rational weights $a_j \in \Q$. Let $m \in\N$ be a common
denominator. Let
$$
\mathcal L_m = \left( (m(\cK_{\cX/S} + {\bf a}))- (m(\cK_{\cX/S} + {\bf
a}))^{-1}\right)^{\otimes 2}.
$$
be an element of the corresponding Grothendieck group. Denote by
$$
\lambda_{m}= \det f_! \mathcal L_m
$$
the determinant line bundle on $S$. The Hirzebruch-Riemann-Roch Theorem
states that the Chern class of the determinant line bundle equals the
degree $2$ component
$$
c_1(\lambda_{m})= - f_* \left(ch(\mathcal L_m) td(X/S)\right)_{(2)} =
4m^2 f_*\left( c_1^2(\cK_{\cX/S} + \bf a)\right)_{(2)}.
$$
Now we equip the $\Q$-bundle $\cK_{\cX/S}+ \bf a$ with a $C^\infty$
hermitian metric of the form $\wt g^{-1}$ with {\em positive} curvature,
and denote by
$$
\wt \omega_{\cX}= \ii \pt\ol\pt \log \wt g = 2\pi c_1(\cK_{\cX/S}+ {\bf a}, \wt g^{-1})
$$
the Chern form. We denote by $ch(\mathcal L_m, \wt g )$ the induced Chern
character form. Only the term of degree zero contributes to the Todd
character form and the metric on $\cX$ need not be specified.

The theorem of Quillen \cite{qui}, Zograf-Takhtajan \cite{TaZo0} and
Bismut-Gillet-Soulé \cite{bgs} states the existence of a Quillen metric
$h_0^Q$ on $\lambda_{m}$ such that for the type $(1,1)$ components the
following holds.
\begin{eqnarray*}
c_1(\lambda_m, h^Q_0)&=& - \int_{\cX/S}ch(\mathcal L, \wt g
)td(\cX/S)_{(1,1)}\\
 &=& 4m^2\int_{\cX/S} c_1(\cK_{\cX/S}+ {\bf a}, \wt
g^{-1})^2_{(1,1)}\\ & =& 16 m^2 \pi^2 \int_{\cX/S} \wt\omega_\cX^2.
\end{eqnarray*}

\begin{theorem}\label{th:quillen}
Let $f:(\cX, \bf a) \to S$ be the universal holomorphic family of weighted
punctured Riemann surfaces equipped with the family $g_{\bf a}$ of conical
metrics. Let $\omega^{WP}$ be the generalized \wp metric. Then the
determinant line bundle $\lambda_m$ possesses a  hermitian metric $h$ of
class $C^\infty$, whose Chern form is up to a numerical factor equal to
the \wp metric:
$$
c_1(\lambda_m, h) = 16 m^2\pi^2 \omega^{WP}.
$$
The metric $h$ descends to the moduli space.
\end{theorem}
Since Hilbert space methods are not available, the notion of an analytic
torsion of Dirac operators is void, in particular there is no Quillen
metric in its original sense.

\begin{proof}
We will use the notation of this section, in particular the metric $\wt g$
on $-(\cK_{\cX/S}+ {\bf a})$. We can chose $\wt g$ invariant under the
\tei modular group. Let $\sigma_\nu$ be the canonical sections of the line
bundles on $\cX$ given by the punctures. These can be chosen as invariant
under the \tei modular group. The quotient
$$
\frac{\wt g}{\Pi_\nu |\sigma_\nu|^{2a_{\nu}}}
$$
is a well-defined relative metric on $\cX$ with poles of fractional order
at the punctures.

In view of Section~\ref{conic} we have
$$
g_{\bf a}=\frac{\wt g}{\Pi_\nu |\sigma_\nu|^{2a_{\nu}}} e^w,
$$
where the function $w$ is globally defined on $\cX$.

Now
$$
\int_{\cX/S} (\omega_{\cX/S}^2 - \wt\omega_{\cX/S}^2) =
\int_{\cX/S} \ii \pt\ol\pt \left(w \cdot(\omega_{\cX/S}+\wt\omega_{\cX/S})\right).
$$
Let the induced relative metric be
$$
\wt\omega_\cX|\cX_s = \wt{\wt g}(z,s) dA.
$$
The assignment
$$
s\mapsto w \cdot\left(\frac{g_{\bf a}}{\wt{\wt g}} + 1 \right)
\wt{\wt g}
$$
defines a $C^{\infty}$ map $S \to L^p$. Now the argument of
the proof of Theorem~\ref{th:kaeh} applies literally, and
$$
\int_{\cX/S} \ii \pt\ol\pt \left(w \cdot(\omega_{\cX/S}+\wt\omega_{\cX/S})\right) =
\ii \pt\ol\pt \int_{\cX/S} \left(w \cdot(\omega_{\cX/S}+\wt\omega_{\cX/S})\right),
$$
where the integral on the right-hand side defines a $C^\infty$ function on
$S$, which is invariant under the \tei modular group.
\end{proof}

\section{Curvature of the Weil-Petersson metric}
In the classical case the Ricci and holomorphic sectional curvatures of
the classical \wp metric were proven to be negative by Ahlfors in
\cite{Ah}. Royden conjectured the precise upper bound for the holomorphic
sectional curvature in \cite{Roy}. The curvature tensor of the \wp metric
for \tei spaces of compact (or punctured) Riemann surfaces was computed
explicitly by Tromba \cite{Tr} and Wolpert \cite{Wo}. In this section we
show the analogous result for the weighted punctured case. Our methods are
different and originate from the higher dimensional case treated in
\cite{Si,Sch}.

We will first explain the approach and notation in the compact case. Let
$f:\cX \to S$ stand for the universal family, and let again $(z,s)$ be
local holomorphic coordinates on $\cX$ with $f(z,s)=s$, where $s^i;
i=1,\ldots, N$ are holomorphic coordinates on $S$. We denote the
coefficients of $\omega_\cX$ by $g(z,s)=g_{z \ol z}(z,s)$, $g_{z
\ol\jmath}$, and $g_{i \ol \jmath}$ resp.\ (cf.\ \eqref{eq:gzz},
\eqref{eq:gsz}, \eqref{eq:gzs}, and \eqref{eq:gss}).

We use the notation of Kähler geometry. Accordingly the Christoffel
symbols are
$$
\Gamma=\Gamma^z_{zz}= \frac{\pt \log g}{\pt z }
$$
and
$$
\Gamma^{\ol z}_{\ol z \ol z}=\ol\Gamma.
$$
The curvature tensor is
$$
R^{z}_{\; z \ol z z}= - g_{z \ol z}.
$$

Our computations require covariant derivatives with respect to the
hyperbolic metrics $g=g(z,s)$ on the fibers $\cX_s$, whereas we can use
ordinary derivatives for parameters. We use the semi-colon notation of the
derivative of any tensor $b$ for both:
$$
\nabla_z b =b_{;z},
$$
and
$$
\nabla_i b= \pt_i b = b_{;i},
$$
where the index $i$ stands for the coordinate $s^i$ so that $\pt_i =
\pt/\pt s^i$.

Let the tangent vectors $(\pt/\pt s^i)|_s$ correspond to harmonic Beltrami
differentials
$$
\mu_i = \mu^z_{i\, \ol z} \pt_z \ol{dz}
$$
with $\mu_{\ol\jmath} = \ol{\mu_j}$

Now the \wp form in coordinates $s^i$ equals
$$
\omega^{WP}_S = \frac{\ii}{2} G_{i \ol\jmath}(s) ds^i \wedge ds^{\ol \jmath},
$$
where
$$
G_{i \ol\jmath}(s) = \langle \mu_i   , \mu_j  \rangle = \int_{\cX_s}\mu_i \mu_{\ol\jmath} g dA.
$$
Like in Lemma~\ref{le:hori} and Proposition~\ref{pr:harmbel} we use the
horizontal lifts
$$
V_i = \pt_i +a_i^z \pt_z,
$$
We set $V_{\ol \jmath}= \ol{V_j}$ and $a_{\ol \jmath}= \ol{a_j}$, i.e.\
$a^{\ol z}_{\ol j} = \ol{a^z_k}$. We have
\begin{equation}\label{eq:defmu}
\mu_i = a^z_{i
;\ol z}\; \pt_z \ol{dz}.
\end{equation}
In order to compute derivatives $\pt_k$ say of the coefficients
$G_{i\ol\jmath}$, in principle we need a differential trivialization of
the family. Instead one can apply the Lie derivative $L_{W_k}$ with
respect to a differentiable lift $W_k$ of the tangent vector $\pt/\pt s^k$
to the integrand. In this way the Lie derivative of the integrand can be
separated into tensors. Also (because of the symmetry of the Christoffel
symbols) we can use covariant derivatives for the computation of Lie
derivatives. As usual, the metric tensor defines a transition from
contravariant to covariant tensors.

As differentiable lifts we take the horizontal lifts $V_k$ described
above. Observe that Lie derivatives are not type preserving.

We will need the following identities.
\begin{eqnarray}
L_{V_k}(g\; dA )&=&0 \label{eq:lvg} \\
\chi_{i \ol \jmath} := \langle V_i, V_j\rangle_{\omega_\cX}
&=& g_{i \ol \jmath} - a^z_i a^{\ol z}_{\jmath} g_{z \ol z} \label{eq:chi}\\
L_{V_k}(\mu_{\ol \jmath}) &=& - (\chi_{k \ol \jmath})^{;\ol z}_{\; ;z} \pt_{\ol z} dz \label{eq:lvmu}\\ && -
(\mu_k)^z_{\; \ol z} (\mu_{\ol \jmath})^{\ol z}_{\; z} \pt_z dz
+ (\mu_k)^z_{\ol z} (\mu_{\ol \jmath})^{\ol z}_{\; z} \pt_{\ol z}\ol{dz}\nonumber
\end{eqnarray}
\begin{proof}
We show \eqref{eq:lvg} and compute the $(z,\ol z)$-component of the Lie
derivative.
$$
(L_{V_k} g_{z\ol z})_{z \ol z}= [\pt_k + a_k^z \pt_z, g_{z \ol z}] = g_{z \ol z;k} +
a^z_k g_{z\ol z;z} + a^z_{k z} g_{z \ol z}= g_{k \ol z; z } + a_{k \ol z;z}=0
$$
The inner product of horizontal lifts in \eqref{eq:chi} with respect to
$\omega_\cX$ was already evaluated for $\dim S=1$ above. Equation
\eqref{eq:lvmu} follows from the \eqref{eq:defmu} and \eqref{eq:chi}.
\end{proof}

\begin{proposition}\label{pr:dGij} For all $s\in S$
\begin{equation}\label{eq:dGij}
\pt_k G_{i\ol\jmath}(s)=\int_{\cX_s} L_{V_k}(\mu_i) \mu_{\ol\jmath} \,g\, dA
\end{equation}
holds.
\end{proposition}
When evaluating \eqref{eq:dGij}, only the first component of
\eqref{eq:lvmu} gives a contribution in the pairing with $\mu_{\ol
\jmath}$.
\begin{proof}
We compute $L_{V_k}(\mu_i \mu_{\ol \jmath}\, g \, dA)$ using
\eqref{eq:lvg}. Now by partial integration (for all $s\in S$):
\begin{equation}\label{eq:lvkj}
\int_{\cX_s} \mu_i L_{V_k}(\mu_{\ol\jmath}) \,g\, dA=- \int_{\cX_s}
{\mu_i}^z_{\; \ol z} {\chi_{i\ol\jmath}}_{ \; ;z}^{;\ol z} \,g\, dA
= \int_{\cX_s}
{\mu_i}^z_{\; \ol z; z} {\chi_{i\ol\jmath}}^{;\ol z} \,g\, dA
=0.
\end{equation}
In the last step we used the harmonicity of $\mu_i$ in the form
\begin{equation}\label{eq:muharm}
  {\mu_i}^z_{\; \ol z; z} = 0.
\end{equation}
\end{proof}
\begin{lemma}\label{le:sym}
\begin{equation}\label{eq:sym}
L_{V_k}(\mu_i)^z_{\; \ol z} = L_{V_i}(\mu_k)^z_{\; \ol z}
\end{equation}
\end{lemma}
The {\em proof} is a direct computation.

We see that Lemma~\ref{le:sym} together with Proposition~\ref{pr:dGij}
also implies the Kähler property.

\begin{lemma}\label{le:lvkij}
The Lie derivatives
$$
L_{V_k}(\mu_i)=L_{V_k}(\mu_i)^z_{\; \ol z} \pt_z \ol{dz}
$$
of the harmonic Beltrami differentials are again harmonic Beltrami
differentials.
\end{lemma}
\begin{proof} We have
\begin{equation*}
\nabla_z  L_{V_k}(\mu_i)=0.
\end{equation*}
Its formal proof corresponds to $\ol\pt^* L_{V_k}(\mu_i)=0$ in
\cite{Sch-3}. The computation is straightforward.
\end{proof}

It is convenient to use normal coordinates of the second kind for the
components of the \wp tensor at a given point $s_0\in S$. Because the
$\mu_i$ span the space of harmonic Beltrami differentials (for $s=s_0$)
the condition
$$
\pt_kG_{i\ol \jmath}(s_0)= 0
$$
by Proposition~\ref{pr:dGij} is equivalent to saying that all derivatives
$L_{V_k}(\mu_i)$ vanish at $s=s_0$ identically.

We compute the second derivative at the given point $s_0$. By
\eqref{eq:dGij}
\begin{equation}\label{eq:dklGij}
\pt_{\ol\ell}\pt_k G_{i\ol\jmath} =
\int_{\cX_{s_0}} L_{V_\ol\ell}L_{V_k}(\mu_i) \mu_{\ol \jmath} \,g \, dA +
\int_{\cX_{s_0}}L_{V_k}(\mu_i)L_{V_\ol\ell}(\mu_{\ol \jmath})\,g \, dA.
\end{equation}
\begin{lemma}\label{le:brac}
\begin{eqnarray}
[V_\ol{\ell},V_k] &=& - \chi_{k\ol \ell}^{\; ;z} \pt_z + \chi_{k\ol\ell}^{\; ;\ol
z}\pt_{\ol z}\label{eq:brac}\\
\int_{\cX_s}L_{[V_\ol{\ell},V_k]}(\mu_i)\mu_{\ol \jmath}\, g \, dA &=&
-\int_{\cX_s} \Delta(\chi_{k\ol \ell})\mu_i\mu_{\ol \jmath}\, g \, dA \label{eq:brac2}
\end{eqnarray}
\end{lemma}

We omit the computational proof of \eqref{eq:brac}. In order to see
\eqref{eq:brac2} we write
$$
[\chi^{;z}_{k\ol\ell}\pt_z, \mu_{i\, \ol z}^z \pt_z \ol{dz}]^z_{\; \ol z}
= -\chi^{;z}_{k\ol\ell} \mu^z_{i \ol z; z} + \chi^{;z}_{k\ol\ell;z} \mu^z_{i \ol z},
$$
where the first term on the right-hand side vanishes because of the
harmonicity of the Beltrami differential $\mu_i$. So we have the
right-hand side of \eqref{eq:brac2}. Finally
$$
[\chi^{;\ol z}_{k\ol\ell}\pt_{\ol z}, \mu_{i\, \ol z}^z \pt_z \ol{dz}]^z_{\; \ol z}
=\chi^{;\ol z}_{k\ol\ell}\mu_{i\, \ol z; \ol z}^z +\chi^{;\ol z}_{k\ol\ell;\ol z}\mu_{i\, \ol z}^z
=  (\chi^{;\ol z}_{k\ol\ell}\mu_{i\, \ol z}^z )_{;\ol z}
$$
so that (again by harmonicity)
$$
[\chi^{;\ol z}_{k\ol\ell}\pt_{\ol z}, \mu_{i\, \ol z}^z \pt_z \ol{dz}]^z_{\; \ol z} \cdot \mu^{\ol z}_{\ol\jmath z}
=  (\chi^{;\ol z}_{k\ol\ell}\mu_{i\, \ol z}^z \mu^{\ol z}_{\ol\jmath z})_{;\ol z}.
$$
The divergence theorem implies that the integral vanishes. \qed

We continue the computation of \eqref{eq:dklGij}.

We use the fact that $L_{V_{\ol\jmath}}L_{V_k}(\mu_i) =
L_{[V_\ol{\ell},V_k]}(\mu_i) +L_{V_k}L_{V_{\ol\jmath}}(\mu_i)$ and apply
Lemma~\ref{le:brac}. Now
\begin{gather}
\pt_{\ol\ell}\pt_k G_{i\ol\jmath} = \int_{\cX_s}
L_{[V_\ol{\ell},V_k]}(\mu_i) \mu_{\ol \jmath}\, g\, dA + \int_{\cX_s}
L_{V_k}L_{V_\ol\ell}(\mu_i) \mu_{\ol \jmath}\, g\, dA \label{eq:curv1} \\
\hspace{7cm}+ \int_{\cX_s} L_{V_k}(\mu_i) L_{V_\ol\ell}(\mu_{\ol
\jmath})\, g\, dA\nonumber
\end{gather}
The third term of \eqref{eq:curv1} vanishes at $s_0$, because for $s=s_0$
in normal coordinates
$$
L_{V_k}(\mu_i)=0.
$$
Now
\begin{gather}
\pt_{\ol\ell}\pt_k G_{i\ol\jmath} = -\int_{\cX_s}\Delta(
\chi_{k\ol\ell})\mu_i\mu_\ol\jmath\, g \, dA + \pt_k \int_{\cX_s}
L_{V_\ol\ell}(\mu_i)\mu_\ol\jmath\, g\, dA\label{eq:curv2}\\
\hspace{3cm}  - \int_{\cX_s} L_{V_\ol\ell}(\mu_i) L_{V_k}(\mu_{\ol
\jmath})\, g\, dA. \nonumber
\end{gather}
In order to treat the {\em first term} of \eqref{eq:curv2}, we use the
equation
\begin{equation}\label{eq:lap}
(-\Delta + id) \chi_{k \ol\ell} =\mu_k \mu_{\ol \ell}
\end{equation}
corresponding to \eqref{laplace}. So that
\begin{gather}
-\int_{\cX_s}\Delta(\chi_{k\ol\ell})\mu_i\mu_\ol\jmath\, g \, dA =
\int_{\cX_s}\Delta(-\Delta+id)^{-1}(\mu_k\mu_\ol\ell)\cdot(\mu_i\mu_\ol\jmath)\,
g \, dA \nonumber \\
= -\int_{\cX_s}\left((-\Delta + id) - id
\right)(-\Delta+id)^{-1}(\mu_k\mu_\ol\ell)\cdot(\mu_i\mu_\ol\jmath)\, g \,
dA\label{eq:chidelta}   \\
=-\int_{\cX_s}(\mu_k\mu_\ol\ell)\cdot(\mu_i\mu_\ol\jmath)\, g \, dA +
\int_{\cX_s}(-\Delta+id)^{-1}(\mu_k\mu_\ol\ell)\cdot(\mu_i\mu_\ol\jmath)\,
g \, dA \nonumber
\end{gather}
The {\em second term} of \eqref{eq:curv2} vanishes by \eqref{eq:lvkj}.

In the {\em third term} of \eqref{eq:curv2} all three components of
\eqref{eq:lvmu} matter. We will use the following identity that follows
from the hyperbolicity of the metrics:
$$
\chi_{k\ol \jmath \; ;zz\ol z}=
\chi_{k\ol \jmath\; ;z\ol z z} -\chi_{k\ol \jmath\; ;z} R^z_{\; z z \ol z} =
\chi_{k\ol \jmath\; ;z\ol z z} -g_{z \ol z}\; \chi_{k\ol \jmath \; ; z}.
$$
So
\begin{gather*}-\int_{\cX_s} (\chi_{i \ol\ell})_{;\ol z\ol
z}(\chi_{k\ol\jmath})_{;zz} (g^{\ol z z})^2  \,g \, dA =
\int_{\cX_s}(\chi_{i \ol\ell})_{;\ol z}(\chi_{k\ol\jmath})_{;zz\ol
z}(g^{\ol z z})^2  \,g \, dA\\
= -\int_{\cX_s}(\chi_{i \ol\ell})_{;\ol z
z}\left((\chi_{k\ol\jmath})_{;z\ol z} - g_{z\ol z}\chi_{k\ol\jmath}
\right)(g^{\ol z z})^2 \,g \, dA = \int_{\cX_s} \Delta(\chi_{i\ol\ell})
\mu_k\mu_{\ol \jmath} \, g\, dA
\end{gather*}
The argument above shows that this is exactly equal to
$$
-\int_{\cX_s}(\mu_i\mu_\ol\ell)\cdot(\mu_k\mu_\ol\jmath)\, g \, dA +
\int_{\cX_s}(-\Delta+id)^{-1}(\mu_i\mu_\ol\ell)\cdot(\mu_k\mu_\ol\jmath)\,
g \, dA.
$$
Hence the third term in \eqref{eq:curv2} contains the three contributions
of \eqref{eq:lvmu}, it equals
\begin{gather}
-\int_{\cX_s} (\chi_{k\ol\jmath})_{;zz} (\chi_{i \ol\ell})_{;\ol z\ol z}
(g^{\ol z z})^2  \,g \, dA \hspace{4cm}\strut \nonumber \\ \hspace{4cm}+
\int_{\cX_s} (\mu_i\mu_\ol \jmath)(\mu_k \mu_\ol\ell ) \,g \, dA +
\int_{\cX_s} (\mu_i\mu_\ol \ell)(\mu_k \mu_\ol\jmath ) \,g \, dA \label{eq:chizzpartial}  \\
= \int_{\cX_s} (-\Delta + id)^{-1}(\mu_k \mu_{\ol\jmath}) (\mu_i
\mu_\ol\ell) + \int_{\cX_s} (\mu_i\mu_\ol \jmath)(\mu_k \mu_\ol\ell ) \,g
\, dA \nonumber
\end{gather}
(Here we gathered the Beltrami differentials in a convenient way.)

Adding all terms together, we have the curvature of the \wp metric.

\begin{theorem}
Let $s^i$ be holomorphic coordinates on the \tei space and let the tangent
vectors $\left.\frac{\pt}{\pt s^i}\right|_{s_0}$ correspond to the
harmonic Beltrami differentials $\mu_i$ on $X=\cX_{s_0}$. Then
\begin{eqnarray}
{R_{i\ol{\jmath}k\ol{\ell}}}(s_0) &=& \int_X \left( \Delta-
id \right)^{-1}(\mu_i \mu_\ol{\jmath})\;\mu_k \mu_\ol{\ell}\; g\, dA \label{eq:curv}  \\
&&+   \int_X \left( \Delta- id \right)^{-1}(\mu_i \mu_\ol{\ell})\;\mu_k \mu_\ol{\jmath}\; g\, dA.\nonumber
\end{eqnarray}
holds.
\end{theorem}

(We have been using the complex Laplacian with non-positive eigenvalues as
opposed to the real one, which accounts for a factor of $2$.)

In the case of the generalized \wp metric for weighted Riemann surfaces we
will show that the same formula holds, for weights larger that $1/2$. This
is the range, where also Fenchel-Nielsen coordinates were introduced. It
contains the interesting range of weights of the form $1-1/m$, $m>2$,
which arise from orbifold singularities.

\begin{theorem}
Let $(X,\A)$ with $1/2 < a_j <1$  be a weighted punctured Riemann surface,
which is represented by a point $s_0$ in the \tei space
$\mathcal{T}_{\gamma,n}$. Let $s^1,\ldots s^N$ be any local holomorphic
coordinates near $s_0$, and let $\mu_{\alpha} \in H^1(X,\A)$ be harmonic
representatives of the vectors $\left.{\frac{\partial}{\partial
s\alpha}}\right|_{s_0}$. Then the curvature tensor of the \wp metric is
given by \eqref{eq:curv}, where the Laplacian and the area elements are
replaced by $\Delta_{\bf a}$ and $dA_{g_{\bf a}}$, which are induced by
the hyperbolic conical metric on the fiber.
\end{theorem}
In all of our arguments we will assume that the anti-holomorphic quadratic
differentials that define harmonic Beltrami differentials have at most a
pole at the given conical singularity, (in the holomorphic case the proofs
are still valid).

We first prove the statement of Proposition~\ref{pr:dGij} in the conical
case.

We need to see that the integration commutes with a differentiation with
respect to the parameter (after a differentiable trivialization of the
family). This follows like in the proof of Theorem~\ref{th:kaeh}:

Let first be $\wt V_k$ be any $C^\infty$ lift. Now the map $s\mapsto \mu_i
\mu_{\ol \jmath} g_{\bf a} $ again is a $C^\infty$ map from $S$ to
$L^h(X)$ by Lemma~\eqref{Sob}(iii). So by our previous argument,
$$
\pt_k G_{i\ol\jmath}(s)=\int_{\cX_s} L_{\wt V_k}(\mu_i \mu_{\ol\jmath} \,g\, dA)
$$
holds. Next, we need that
$$
V_k - \wt V_k = C^z \pt_z
$$
is a (global) tensor in fiber direction. Now
\begin{gather*}
[C^z \pt_z, (\mu_i\mu_{\ol \jmath}) g_{z\ol z} \ii dz \wedge \ol{dz}] \hspace{4cm} \strut \\
= C^z \pt_z((\mu_i\mu_{\ol \jmath}) g_{z\ol z})+ \pt_z(C^z)
((\mu_i\mu_{\ol \jmath}) g_{z\ol z})\ii dz \wedge \ol{dz} \\
= d(\ii  C^z\cdot (\mu_i\mu_{\ol \jmath}) g_{z\ol z}\ol{dz}) = d(C_{\ol
z}(\mu_i\mu_{\ol \jmath})).
\end{gather*}
{\em Claim.}
$$
\int_{\cX_s} d(C_{\ol z}(\mu_i\mu_{\ol \jmath})) =0.
$$

{\em Proof of the Claim.} We write the above integral as limit of
integrals over closed paths around the punctures. We first estimate the
coefficient $C_{\ol z}$. It satisfies the same estimates like the $a_{k\ol
z}$. Now
$$
\frac{\pt a_k^z}{\pt \ol z} = \mu_k = \frac{\ol \varphi}{g}
$$
for some holomorphic quadratic differential $\varphi$ with at most a
simple pole. Because of Remark~\ref{re:grealana} we can find a continuous
$\ol z$-anti-derivative $\eta$ of the right-hand side on a punctured disk
$U^*$. This fact follows from the more general Remark~\ref{re:} below.

The term
$$
a_k^z -\eta
$$
is holomorphic on $U^*$ and
$$
a_{k\ol z}= g \cdot a_k^z \in H^p_1(U) \subset L^2(U)
$$
for $p<1/a$ by Theorem~\ref{smooth}. In particular
$$
a_k^z - \eta \in L^2(U).
$$
Hence $a_k^z -\eta$ must be holomorphic at the puncture (cf.\ \cite[Lemma
5.3]{S-T}. Now
$$
|C_{\ol z}| \simeq |a_{k\ol z}|= |g \cdot a_k^z| \lesssim |z|^{-2a},
$$
and
$$
|C_z (\mu_i\mu_{\ol \jmath})| \lesssim \frac{1}{|z|^{2(1-a)}}
$$
so that
$$
\lim_{\epsilon \to 0} \int_{r=\varepsilon} \frac{rd\varphi}{r^{2(1-a)}} =0
$$
implies the claim. \qed

\begin{lemma}\label{le:anti}
Let $\Delta_r$ the disk of radius $r$ in $\C,$ and $1/2 < \alpha < 1$. Let
$f  \in C^{\infty}(\Delta_1 \backslash \{0 \},\C)$ be a function such that
$|z|^{2 \alpha} f(z) $ is bounded in a neighborhood of $0$. Let $U$ be a
relatively compact open subset of $\Delta_1$ containing $0$. Then the
equation
\begin{equation}\label{eq:dbar}
\frac{\partial g}{\partial \bar{z}}  = f
\end{equation}
has a solution which is of class $C^{\infty}$ on $(U \backslash \{0  \})$
and such that $|z|^{2 \alpha -1} g(z) $ is bounded in a neighborhood of
$0$. In particular it is contained in $L^2(U)$. Moreover any solution of
\eqref{eq:dbar}, which is in $L^2(U)$ has this boundedness property.
\end{lemma}
\begin{proof}
For any $0\leq r <\rho<1$, and $\Delta_{r,\rho} = \{ z \in \C : r < |z| <
\rho \}$. We define
$$
F(r,\rho)(z)=\frac{-1}{ \pi } \int_{ \Delta_{r,\rho}} \frac{f(\zeta)}{ \zeta - z}  \  i
\frac{d \zeta  \wedge d \overline{\zeta}}{2 }.
$$
It is known \mapa{?? quote some reference? e.g. Kodaira Morrow?} that for
$z\in \Delta_{r,\rho}$ with $r>0$
$$
\frac{\pt F(r,\rho)(z)}{\ol{\pt z}}=f(z)
$$
holds. Let $K$ be a compact subset of  $ \Delta_{\rho} \backslash  \{0
\}$. Let $r_0>0$ be chosen such that $K \subset \Delta_{r_0,\rho}$. Then,
for all $0<r<r_0$, the function $|f(\zeta)/|\zeta-z|$ is uniformly bounded
by some $M>0$ for $z\in K$ and $\zeta \in \Delta_{\rho_0}\backslash
\{0\}$. So
$$
|F(r,\rho)(z)-F(0,\rho)(z)| \leq M r^2,
$$
which implies uniform convergence for $r\to 0$. The same argument holds
for the derivatives with respect to $z$ and $\ol z$. It follows that
$F(0,\rho)$ is differentiable and solves \eqref{eq:dbar} on
$\Delta_\rho\backslash\{0\}$.

On the open set $\Delta_\rho\backslash\{0\}$ we write
$f(\zeta)=|\zeta|^{-2\alpha}m(\zeta)$ with $|m(\zeta)|\leq C$. Now  we
make the change of variables  $\zeta = z \eta$. Then
$$
|F(0,\rho)(z)|
\leq  \frac{1}{ \pi } C |z|^{-2 \alpha+1}  \int_{ \C} \frac{1}{| \eta|^{2
\alpha}|( \eta-1)|} \  i \frac{d \eta  \wedge d \overline{\eta}}{2 }.
$$
so that  we just need to show that
$$
\int_{ \C} \frac{1}{| \eta|^{2 \alpha}|( \eta-1)|} \
  i \frac{d \eta  \wedge d \overline{\eta}}{2 } < +
\infty.
$$
The convergence of the right-hand side integral follows from
$$
\int_0^1 \frac{r dr}{r^{2\alpha}} < \infty  \text{\quad and \quad}
\int_2^\infty \frac{rdr}{r^{2\alpha+1}} < \infty.
$$
Choose $\rho$ such that $\overline{U} \subseteq \Delta_{\rho}$. Then the
condition $|z|^{2 \alpha-1}g(z)$ bounded implies that  $g$  is in
$L^2(U).$ The second claim now follows from the fact that a holomorphic
function in $\Delta_{\rho} \backslash \{ 0  \}$, which is in $L^2(U)$ is
bounded on $U$.
\end{proof}
Applying the above Lemma applied to $g(z)/z$ we can treat the case
$0<\alpha<1/2$.
\begin{remark}\label{re:}
An analogous statement holds for $0<\alpha<1/2$: Under the same
boundedness assumption, the equation \eqref{eq:dbar} has a solution, which
extends continuously to the origin.
\end{remark}

\begin{remark}
If $ \alpha = \frac{1}{2}$ a statement as in  the Lemma does not hold in
general.
\end{remark}
In fact, let us choose $f(z) = 1/\overline{z} = (\partial/\partial
\overline{z}) \log(|z|^2)$. Assume that there exists a bounded function
$g$ on a small punctured disk such that $g - \log(|z|^2)$ is holomorphic.
Since  $g - \log(|z|^2)$ is in $L^2$ of the disk, we would obtain that it
is bounded. However $ \log(|z|^2)$ is not bounded near $0$. By replacing
$f$ by $z^k f$ and $g$ by $z^{-k} g$ for some integer $k$, we may prove a
similar lemma for $\alpha \in \R $ such that $ 2 \alpha \neq \Z,$ and find
an example as above if $2 \alpha \in \Z$. \qed

We return to the discussion of the generalized \wp metric. We know that
$$
\pt_k G_{i\ol\jmath}(s)=\int_{\cX_s} L_{V_k}(\mu_i \mu_{\ol\jmath} \,g\, dA)
$$
So far the integral can be computed in terms of the (singular) horizontal
lifts $V_k$, and we are in a position to also use covariant derivatives,
since the Lie derivatives can also be computed in terms of those.

We use the fact $L_{V_k}(g)=0$ from \eqref{eq:lvg}, which is still
pointwise true outside a set of measure zero so that the statement of
Proposition\ref{pr:dGij} in the conical case is reduced to showing that
$$
\int_{X} \mu_i L_{V_k}(\mu_{\ol\jmath})g_{\bf a}dA =0.
$$
By \eqref{eq:lvmu} the above integral equals
$$
- \int_{X} {\mu_i}^z_{\; \ol z} {\chi_{k\ol\jmath}}_{ \; ;z}^{;\ol z} \,g\, dA=
- \int_{X} \left({\mu_i}^z_{\; \ol z} {\chi_{k\ol\jmath}}^{;\ol z}\right)_{;z} \,g\, dA
$$
because of the harmonicity of $\mu_i$. This integral is up to a numerical
factor written as
$$
\int_X d\left(\mu_i \frac{\pt}{\pt z} \ol{dz} \cdot \chi_{k\ol\jmath;z} dz  \right)=
\int_X d\left(\mu_i \cdot \chi_{k\ol\jmath;z} \ol{dz}  \right)
$$
We consider the defining equation for $\chi=\chi_{k \ol \jmath}$ in the
form
$$
\frac{\pt^2 \chi}{\pt z \ol{\pt z}} = - g \mu_k \mu_{\ol \jmath}  + g \chi.
$$
We know that $\chi$ is continuous and for some neighborhood $U$ of a
puncture, again, we apply Lemma~\ref{le:anti} and take a $\ol
z$-anti-derivative $\eta$ of the right-hand side, which satisfies (with
$a> 1/2$)
$$
|\eta| \lesssim r^{1-2a}.
$$
Again, since $\pt\chi/ \pt z \in H^p_1 \subset L^2$, the function
$$
\frac{\pt \chi}{\pt z} - \eta \in \mathcal O(U^*)
$$
is holomorphic at the puncture. With the estimate for $\mu_i$ the argument
of the previous claim immediately yields the vanishing of the integral.

Next, we chose normal coordinates for the \wp metric on $S$ of the second
kind at a given pint $s=s_0$.

This concludes the proof of Proposition~\ref{pr:dGij} in the conical case.

We will follow the computation of the curvature of the \wp metric in the
compact case.

\begin{lemma}\label{le:lvkijconic}
The Lie derivatives
$$
L_{V_k}(\mu_i)^z_{\; \ol z} \pt_z \ol{dz}
$$
of the harmonic Beltrami differentials are again harmonic Beltrami
differentials with respect to the conical structure (depending in a
$C^\infty$-way upon the parameter).
\end{lemma}
Now we can apply the argument of Proposition~\ref{pr:dGij} literally to
$$
\int_{\cX_s} L_{V_k}(\mu_i) \mu_{\ol\jmath} \,g\, dA
$$
and get
\begin{corollary}
Equation \eqref{eq:dklGij} holds in the conical case.
\end{corollary}

\begin{proof}[Proof of Lemma~\ref{le:lvkijconic}] We have from Lemma~\ref{le:lvkij}
\begin{equation}\label{eq:dbsm}
\nabla_z  L_{V_k}(\mu_i)=0.
\end{equation}
We need to see that the anti-holomorphic term $g \cdot L_{V_k}(\mu_i)$ is
in $L^1$ so that it can have at most a simple pole. After the
verification, we know that $L_{V_k}(\mu_i)$ is a harmonic Beltrami
differential in our sense. We have
\begin{gather*}
g\cdot L_{V_k}(\mu_i)^z_{\; \ol z} = - g \pt_z\left(
\frac{1}{g}\pt_k(g_{i\ol z})\right)- g_{i \ol z} g_{k\ol z}\\ = - \pt_{\ol
z}(\log g)(\pt_k g_{i\ol z})+ \pt_{\ol z}(\pt_k g_{i\ol z}) - g_{i \ol z}
g_{k\ol z}.
\end{gather*}
We show that all three terms are in $L^1$: For the first term we can use
the proof of Lemma~\ref{le:phiL1}. Since $g_{i\ol z} \in H^p_1$ the second
term is in $L^1$. Finally both $g_{i\ol z}$ and $g_{k\ol z}$ are in
$H^p_1\subset L^2$.
\end{proof}

We prove the statement of Lemma~\ref{le:brac} in the conical case: The
equation \eqref{eq:brac} is pointwise and carries over. We show
\eqref{eq:brac2}: For the required partial integration we just need that
$$
\int_{\cX_s} d(\chi_{k \ol \ell ; z} \mu_i \mu_{\ol\jmath}dz)=0.
$$
As above we reduce this to the vanishing of limits of integrals along
closed paths around the punctures.

We know from Lemma~\ref{le:chiz} below that $ |\chi_{k \ol\ell\, : z}|
\lesssim r^{-2a+1} $ and $\mu_i\mu_\ol\jmath \sim r^{-2+4a}$ so that
$|\chi_{k \ol \ell ; z} \mu_i \mu_{\ol\jmath}| \lesssim r^{4a-1}$. So
$$
\lim_{r\to 0}( r\cdot r^{4a-1}) =0
$$
implies that the above integral vanishes, which proves \eqref{eq:brac2} in
the conical case.\qed

In particular \eqref{eq:curv2} is now valid in our situation. A purely
local computation (under the integral sign) implies \eqref{eq:chidelta}.

The final step is to arrive at \eqref{eq:chizzpartial} in the conical
case, i.e.\ to apply a twofold partial integration to
$$
\int_{\cX_s} (\chi_{i \ol\ell})_{;\ol z\ol z}(\chi_{k\ol\jmath})_{;zz}
(g^{\ol z z})^2  \,g \, dA.
$$
This is achieved by the following Lemmas \ref{le:chizolz} and
\ref{le:chizz}.

\begin{lemma}\label{le:chizolz}
The following singular integral vanishes.
\begin{equation}\label{eq:chizolz}
\int_X d(\chi_{k\ol\jmath\, ; z\ol z} \chi_{i\ol\ell\, ; \ol z} g^{\ol z
z} \ol{dz}) =0.
\end{equation}
\end{lemma}
\begin{proof}
The integrand equals
\begin{gather*}
d\left((\chi_{k\ol\jmath} - \mu_{k}\mu_{\ol \jmath}  )\chi_{i \ol\ell\, :
\ol z}\ol{dz}\right).
\end{gather*}
Now we know from Lemma~\ref{le:chiz} below that
$$
|\chi_{i \ol\ell\, :\ol z}| \lesssim r^{-2a+1}
$$
and have the continuity of $\chi_{k\ol\jmath}$. Furthermore
$$
|\mu_{k}\mu_{\ol \jmath}|\lesssim r^{-2+4a}
$$
so that $\chi_{k\ol\jmath} - \mu_{k}\mu_{\ol \jmath}$ is continuous. We
use integration along closed loops as above and see that the integral
vanishes.
\end{proof}

\begin{lemma}\label{le:chizz}
\begin{equation}\label{eq:chizz}
\int_X d(\chi_{k\ol\jmath\, ; z z} \chi_{i\ol\ell\, ; \ol z} g^{\ol z
z} {dz}) =0.
\end{equation}
\end{lemma}

We reduce the proof to the following statement, which shows that possible
residues in \eqref{eq:chizz} and \eqref{eq:chizolz} must be equal up to a
sign. However, we know already that the latter integral vanishes.

\begin{lemma}\label{le:intddb}
\begin{equation}\label{eq:intddb}
\lim_{\varepsilon \to 0}\int_{\{|z|<\epsilon  \}} \frac{\ii}{2}
\pt \ol \pt (\chi_{k\ol\jmath\, ; z}\chi_{i\ol\ell\, ; \ol z} g^{\ol z z} ) =0
\end{equation}
\end{lemma}

\begin{proof}[Proof of \eqref{eq:chizz}]
We expand the integrand of \eqref{eq:intddb} and find:
$$
0=\lim_{\epsilon\to 0} \int_{|z|=\epsilon} \pt\left(\chi_{k\ol\jmath\, ; z}
\chi_{i \ol\ell\, ; \ol z} g^{\ol z z} \right)
= \lim_{\epsilon\to 0} \int_{|z|=\epsilon} \left(\chi_{k\ol\jmath\, ; zz}\chi_{i \ol\ell\, ; \ol z} +
\chi_{k\ol\jmath\, ; z}\chi_{i \ol\ell\, ; \ol z z} \right)g^{\ol z z} dz
$$

\end{proof}

\begin{lemma}\label{le:chiz}
\begin{equation}\label{eq:chiz}
|\chi_{k\ol \jmath\, ;z}| \lesssim  r^{-2a+1}
\end{equation}
\end{lemma}

\begin{proof}
We know that
$$
\pt\ol\pt \chi_{k\ol\jmath}= (\chi_{k\ol\jmath} - \mu_k\mu_\ol\jmath) g.
$$
Furthermore
$$
\chi_{k \ol \jmath} - \mu_k \mu_{\ol\jmath}
$$
is continuous since $\chi_{k \ol \jmath} \in H^p_2(X)$ and the continuity
of $\mu_k \mu_{\ol\jmath}$ follows since $a> 1/2$. Now the argument
involving the $\ol z$-anti-derivative again gives the claim.
\end{proof}

{\em Proof of Lemma~\ref{le:intddb} .} At this point, we assume that
$k=i=j =\ell$. This case is sufficient, because the curvature formula will
follow as usual from this case by polarization. We set $\chi
=\chi_{k\ol\jmath}$  and $|\mu|^2= \mu_i\mu_{\ol\jmath}$ for short. The
integrand equals
\begin{gather*}
\eta := \frac{\ii}{2}\pt \ol\pt (g^{\ol z z} \chi_{;z}\chi_{;\ol z})  \hspace{7cm} \\
= g^{\ol z z} \left( \chi_{; z \ol z z}\chi_{\ol z} + \chi_{; z \ol
z}\chi_{\ol z z} + \chi_{;zz}\chi_{;\ol z\ol z} + \chi_{; z}\chi_{; \ol
z\ol z z} \right) \frac{\ii}{2} dz \wedge \ol{dz}.
\end{gather*}
Now we use \eqref{laplace} i.e.\ \eqref{eq:laplace} on the integrand and
the following formula
$$
\chi_{;\ol z\ol z z} = \chi_{;\ol z z \ol z } - \chi_{;\ol z} R^{\ol z}_{\; \ol z \,\ol z z}
= \chi_{;\ol z z \ol z } - g_{z \ol z} \chi_{;\ol z} = - g_{z \ol z} (|\mu|^2)_{; \ol z}.
$$
Hence
\begin{gather*}
\eta = \Big( g_{z \ol z}  (\chi-|\mu|^2)^2 + g^{\ol z z} \chi_{; z
z}\chi_{; \ol z \, \ol z}  \hspace{5cm}\\  + (\chi-|\mu|^2)_{;z}
(\chi-|\mu|^2 )_{; \ol z} - (|\mu|^2)_{; z}(|\mu|^2)_{; \ol z}\Big)
\frac{\ii}{2}
dz\wedge\ol{dz} \\
\geq - (|\mu|^2)_{; z}(|\mu|^2)_{; \ol z} \frac{\ii}{2} dz\wedge\ol{dz}.
\end{gather*}
Again we realize a harmonic Beltrami differential as a quotient of an
anti-holomorphic quadratic differential with at most a single pole by the
metric tensor and use again the analyticity property of
Remark~\ref{re:grealana}. This implies
$$
|(|\mu|^2)_{;z}| \lesssim r^{4a-3}
$$
so that
$$
(|\mu|^2)_{;z} (|\mu|^2)_{;\ol z} \lesssim r^{8a -6}.
$$
So for some $c, r_0>0$  and all $0<|z| \leq r_0$ , we have
$$
\frac{\ii}{2}\pt\ol\pt\Big( g^{\ol z z} \chi_{;z}\chi_{;\ol z} - c\cdot r^{8a -4}\Big)\geq 0.
$$
Observe that by our assumption $r^{8a -4} \to 0$ with $r\to 0$. We write
$(\ii/2)\pt\ol\pt \tau$ for the above expression. In terms of polar
coordinates $z= r \cdot exp(\ii\varphi)$ we set
$$
\wt \tau (r) = 
\int_0^{2\pi} \tau(r,\varphi)d\varphi.
$$
Hence for all $0<\delta<\epsilon$ we have
$$
0\leq \int_{\delta<|z|<\epsilon} \frac{\ii}{2} \pt\ol\pt \tau =
\int_\delta^\epsilon \frac{\pt}{\pt r}\left( r \cdot \frac{\pt}{\pt r} \wt \tau \right) dr
= r \cdot \frac{\pt}{\pt r} \wt \tau \Big|_\delta^\epsilon.
$$
Up to a multiplicative constant the contribution of $-c \cdot r^{8a -4}$
to the integral amounts to
$$
r^{8a-4}\big|^\epsilon_\delta,
$$
which tends to zero with $\epsilon,\delta \to 0$.

The monotonicity implies the existence of
$$
 \ell =\lim_{r \to 0} r \frac{\pt\wt \tau }{\pt r} \geq -\infty.
$$
If we assume $\ell \leq -c' <0$, we see immediately that
$$
\wt\tau(r)\geq c'' - c'\log r
$$
for some real number $c''$ so that $\wt \tau \to \infty$ with $r\to 0$. On
the other hand, it follows from \eqref{eq:chiz} that
$$
\chi_{;z}\chi_{;\ol z} g^{\ol z z}\lesssim r^{2-2a}.
$$
So
$$
\lim_{r\to 0}  r \frac{\pt\wt \tau }{\pt r} \geq 0
$$
is a finite number and
\begin{gather*}
\lim_{\tiny\vbox{\hbox{$\epsilon \to 0$} \hbox{$\delta \to
0$}\hbox{$\epsilon>\delta$}}} \int_{\delta<|z|<\epsilon}
\frac{\ii}{2}\pt\ol\pt\left(\|\chi_{;z}\|^2 \right)=
\lim_{\tiny\vbox{\hbox{$\epsilon \to 0$} \hbox{$\delta \to
0$}\hbox{$\epsilon>\delta$}}} \int_\delta^\epsilon
\frac{1}{r}\frac{\pt}{\pt r}\left(r\cdot \frac{\pt\tau}{\pt r}\right) r dr\\
\hspace{5cm} = \lim_{\tiny\vbox{\hbox{$\epsilon \to 0$} \hbox{$\delta \to
0$}\hbox{$\epsilon>\delta$}}}
 \left.\left( r \cdot \frac{\pt\wt \tau}{\pt r}  \right)\right|^\epsilon_\delta =0.
\end{gather*}
\qed

Georg Schumacher

\vskip .5mm Fachbereich Mathematik und Informatik der Philipps-Universit{\"a}t,
\vskip .5mm Hans-Meerwein-Strasse, Lahnberge, \vskip .5mm D-35032 Marburg,
Germany \vskip .5mm {\tt schumac@mathematik.uni-marburg.de}

\bigskip

Stefano Trapani \vskip .5mm Dipartimento di Matematica, Universit\'a di Roma
Tor Vergata \vskip .5mm Via della Ricerca Scientifica, I-00133 Roma Italy
\vskip .5mm {\tt trapani@mat.uniroma2.it}

\end{document}